\documentclass[11pt,a4paper]{article}
\pdfoutput=1
\usepackage{hyperref}
\hypersetup{hypertexnames = false, bookmarksdepth = 2, bookmarksopen = true, colorlinks, linkcolor = black, citecolor = black, urlcolor = black, pdfstartview={XYZ null null 1}}
\usepackage[utf8]{inputenc}

\usepackage{fullpage}

\usepackage{amsfonts}
\usepackage[fleqn, leqno]{amsmath}
\usepackage{amsthm}
\usepackage{mathdots}

\usepackage[capitalise,noabbrev]{cleveref}

\usepackage[backend=biber, maxbibnames=99, datamodel=mrnumber, sortcites]{biblatex}

\usepackage{booktabs}
\usepackage{colonequals}
\usepackage{diagbox}
\usepackage{enumitem}
\usepackage{mathtools}
\usepackage{parskip}
\usepackage{thmtools}
\usepackage{tikz-cd}
\usepackage[colorinlistoftodos, textsize = footnotesize]{todonotes}
\usepackage{xparse}
\usepackage{xspace}

\usepackage[T1]{fontenc}
\usepackage{libertine}
\usepackage[libertine]{newtxmath}
\usepackage[scaled=0.83]{beramono}
\usepackage{eucal}
\usepackage{microtype}
\usepackage{gitinfo2}

\setlist[enumerate,1]{label=(\roman*),ref=(\roman*),font=\normalfont}

\DeclareFieldFormat{mrnumber}{%
  MR\addcolon\space
  \ifhyperref
    {\href{http://www.ams.org/mathscinet-getitem?mr=#1}{\nolinkurl{#1}}}
    {\nolinkurl{#1}}}

\renewbibmacro*{doi+eprint+url}{%
  \iftoggle{bbx:doi}
    {\printfield{doi}}
    {}%
  \newunit\newblock
  \printfield{mrnumber}%
  \newunit\newblock
  \iftoggle{bbx:eprint}
    {\usebibmacro{eprint}}
    {}%
  \newunit\newblock
  \iftoggle{bbx:url}
    {\usebibmacro{url+urldate}}
    {}}

\newcounter{todocounter}
\DeclareDocumentCommand\addreference{g}{\stepcounter{todocounter}\todo[color = blue!30]{\thetodocounter. Add reference\IfNoValueF{#1}{: #1}}\xspace}
\DeclareDocumentCommand\checkthis{g}{\stepcounter{todocounter}\todo[color = red!50]{\thetodocounter. Check this\IfNoValueF{#1}{: #1}}\xspace}
\DeclareDocumentCommand\fixthis{g}{\stepcounter{todocounter}\todo[color = orange!50]{\thetodocounter. Fix this\IfNoValueF{#1}{: #1}}\xspace}
\DeclareDocumentCommand\expand{g}{\stepcounter{todocounter}\todo[color = green!50]{\thetodocounter. Expand\IfNoValueF{#1}{: #1}}\xspace}

\declaretheoremstyle[
  spaceabove = 3pt,
  spacebelow = 3pt,
  bodyfont = \itshape,
]{first}
\declaretheoremstyle[
  spaceabove = 3pt,
  spacebelow = 3pt,
]{second}
\declaretheorem[numberwithin=section, style=first]{theorem}

\declaretheorem[sibling=theorem, style=first]{corollary}
\declaretheorem[sibling=theorem, style=first]{lemma}
\declaretheorem[sibling=theorem, style=first]{proposition}

\declaretheorem[sibling=theorem, style=second]{example}
\declaretheorem[sibling=theorem, style=second]{remark}
\declaretheorem[sibling=theorem, style=second]{definition}

\declaretheorem[sibling=theorem, style=second]{assumption}

\crefname{assumption}{Assumption}{Assumptions}

\declaretheorem[numberwithin=section, style=first, title=Theorem]{alphatheorem}
\declaretheorem[sibling=alphatheorem, style=first, title=Conjecture]{alphaconjecture}

\crefname{alphatheorem}{Theorem}{Theorems}
\crefname{alphaconjecture}{Conjecture}{Conjectures}
\crefname{alphacorollary}{Corollary}{Corollaries}
\crefname{alphaproposition}{Proposition}{Propositions}

\mathchardef\mhyphen="2D
\newcommand\dash{\nobreakdash-\hspace{0pt}}

\DeclareMathOperator\Aut{Aut}
\DeclareMathOperator\codim{codim}
\DeclareMathOperator\dimvect{\underline{dim}}

\DeclareMathOperator\derived{\mathbf{D}}
\DeclareMathOperator\dual{D}
\DeclareMathOperator\End{End}
\DeclareMathOperator\Ext{Ext}
\DeclareMathOperator\HH{H}
\DeclareMathOperator\hochschild{HH}
\DeclareMathOperator\Hom{Hom}
\DeclareMathOperator\RHom{\mathbf{R}Hom}
\DeclareMathOperator\Lie{Lie}
\DeclareMathOperator\Out{Out}
\DeclareMathOperator\Pic{Pic}
\DeclareMathOperator\rk{rk}
\DeclareMathOperator\RsheafHom{\mathbf{R}\mathcal{H}om}
\DeclareMathOperator\sheafHom{\mathcal{H}om}
\DeclareMathOperator\sheafExt{\mathcal{E}xt}
\DeclareMathOperator\source{s}

\DeclareMathOperator\target{t}

\newcommand\bounded{\ensuremath{\mathrm{b}}}
\newcommand\Chi{\ensuremath{\mathrm{X}}}
\newcommand\can{\ensuremath{\mathrm{can}}}
\newcommand\field{\ensuremath{\mathbf{k}}}
\newcommand\GL{\ensuremath{\mathrm{GL}}}
\newcommand\Gm{\ensuremath{\mathbb{G}_{\mathrm{m}}}}

\newcommand\op{\ensuremath{\mathrm{op}}}

\newcommand\tangent{\ensuremath{\mathrm{T}}}

\DeclareDocumentCommand\modulistack{om}{\IfNoValueTF{#1}{\mathcal{M}{(#2)}}{\mathcal{M}^{#1}(#2)}}
\DeclareDocumentCommand\modulispace{om}{\IfNoValueTF{#1}{\mathrm{M}{(#2)}}{\mathrm{M}^{#1}(#2)}}
\DeclareMathOperator\moduli{M}

\DeclareMathOperator\Rep{R}

\DeclareDocumentCommand\representationvariety{om}{\IfNoValueTF{#1}{\mathrm{R}({#2})}{\mathrm{R}^{#1}{(#2)}}}
\newcommand\group[1]{\mathrm{G}_{#1}}

\newcommand\semistable[1]{#1\mhyphen\mathrm{sst}}
\newcommand\stable[1]{#1\mhyphen\mathrm{st}}

\title{Vector fields and admissible embeddings \\ for quiver moduli}
\author{Pieter Belmans \and Ana-Maria Brecan \and Hans Franzen \and Markus Reineke}

\begin{document}
\maketitle

%\gitfootnote{commit: \texttt{\gitAbbrevHash}\hfil date: \texttt{\gitAuthorIsoDate}\hfil \texttt{\gitReferences}}

\begin{abstract}
  We introduce a double framing construction for moduli spaces of quiver representations.
  It allows us to reduce certain sheaf cohomology computations involving the universal representation,
  to computations involving line bundles,
  making them amenable to methods from geometric invariant theory.
  We will use this to show that
  in many good situations
  the vector fields on the moduli space
  are isomorphic as a vector space
  to the first Hochschild cohomology of the path algebra.
  We also show that considering the universal representation
  as a Fourier--Mukai kernel in the appropriate sense
  gives an admissible embedding of derived categories.
\end{abstract}

\tableofcontents

\section{Introduction}
The universal object on a fine moduli space
allows us to probe the geometry of the moduli space.
We will apply this principle to moduli spaces of quiver representations,
in order to describe their (infinitesimal) symmetries,
and to show that the universal object provides an admissible embedding
of the derived category of the path algebra
into the derived category of the moduli space.

Let~$Q$ be an acyclic quiver,
with~$\mathbf{d}$ an indivisible dimension vector
and~$\theta$ a stability parameter,
such that~$\mathbf{d}$ is~$\theta$-amply stable
(i.e., the unstable locus is of codimension at least~2)
and stability agrees with semistability.
These standing assumptions are codified in \cref{assumption:standing-assumption}.
Then the moduli space~$X\colonequals\modulispace[\stable\theta]{Q,\mathbf{d}}$
of~$\theta$-(semi)stable representations of dimension vector~$\mathbf{d}$
is a smooth projective variety,
which comes equipped with a universal object~$\mathcal{U}=((\mathcal{U}_i)_{i \in Q_0},(\mathcal{U}_a)_{a \in Q_1})$,
which is a representation of~$Q$ with values in vector bundles on~$X$.
We can interpret $\mathcal{U}$ as a left $\mathcal{O}_XQ$-module
by identifying it with $\bigoplus_{i \in Q_0} \mathcal{U}_i$
equipped with the structure of a left $\field Q$-module given by the morphisms $\mathcal{U}_a$.

We denote~$e_i$ the idempotent associated to~$i\in Q_0$,
so that~$e_j\field Qe_i$ is the vector space spanned by paths from~$i$ to~$j$.
We define the morphism
\begin{equation}
  \label{equation:local-morphism}
  H_{i,j}^{\mathcal{U}}\colon e_j\field Qe_i\to\sheafHom(\mathcal{U}_i,\mathcal{U}_j)=\mathcal{U}_i^\vee\otimes\mathcal{U}_j:p\mapsto\mathcal{U}_p
\end{equation}
where~$p$ is an oriented path~$a_\ell\cdots a_1$ from~$i$ to~$j$,
and~$\mathcal{U}_p$ is the composition~$\mathcal{U}_{a_\ell}\circ\cdots\circ\mathcal{U}_{a_1}\colon\mathcal{U}_i\to\mathcal{U}_j$.
Taking global sections we obtain the morphism
\begin{equation}
  \label{equation:global-sections}
  h_{i,j}^{\mathcal{U}}\colon e_j\field Qe_i\to\Hom(\mathcal{U}_i,\mathcal{U}_j)\cong\HH^0(X,\mathcal{U}_i^\vee\otimes\mathcal{U}_j).
\end{equation}
The main (technical) result of this paper is the following.
\begin{alphatheorem}
  \label{theorem:global-sections}
  Let~$Q$, $\mathbf{d}$ and~$\theta$ be as in \cref{assumption:standing-assumption},
  and consider the universal representation $\mathcal{U}$
  on~$X=\modulispace[\stable\theta]{Q,\mathbf{d}}$.
  Then
  \begin{enumerate}
    \item for all~$i,j\in Q_0$ the linear map~$h_{i,j}^{\mathcal{U}}$ from \eqref{equation:global-sections}
      is an isomorphism;
    \item the direct sum over~$i,j\in Q_0$ of the isomorphisms~$h_{i,j}^\mathcal{U}$
      induces an isomorphism of algebras
      \begin{equation}
        \label{equation:algebra-isomorphism}
        h^\mathcal{U}\colon\field Q\overset{\simeq}{\to}\End_X(\mathcal{U}).
      \end{equation}
  \end{enumerate}
\end{alphatheorem}

We will use \cref{theorem:global-sections} to
describe vector fields on quiver moduli,
and prove that the universal representation gives an admissible embedding of derived categories.

To obtain these applications,
we will build upon the main result of \cite{rigidity},
which requires a slightly stronger condition than just~$\theta$-ample stability,
called~$\theta$-strong ample stability,
which will be defined in \cref{definition:strongly-amply-stable}.

\paragraph{Vector fields}
The first application of \cref{theorem:global-sections}
is a recipe to compute vector fields on~$X=\modulispace[\stable\theta]{Q,\mathbf{d}}$,
i.e., a description of~$\HH^0(X,\tangent_X)$,
as a measure of the symmetry group of~$X$.
\begin{alphatheorem}
  \label{theorem:vector-fields}
  Let~$Q$, $\mathbf{d}$ and~$\theta$ be as in \cref{assumption:standing-assumption},
  and assume in addition that~$\mathbf{d}$ is~$\theta$-strongly amply stable.
  Consider~$X=\modulispace[\stable\theta]{Q,\mathbf{d}}$.
  There exists the exact sequence
  \begin{equation}
    \label{equation:vector-fields-presentation}
    0 \to \field  \xrightarrow{\phi} \bigoplus_{i \in Q_0} e_i\field Qe_i \xrightarrow{\psi} \bigoplus_{a \in Q_1} e_{\target(a)}\field Qe_{\source(a)} \to \HH^0(X,\tangent_X) \to 0,
  \end{equation}
  where the maps $\phi$ and $\psi$ are defined as
  \begin{align}
    \phi(z) &= z\sum_{i \in Q_0}e_i \label{equation:phi} \\
    \psi\left( \sum_{i \in Q_0} z_ie_i \right) &= \sum_{a \in Q_1} (z_{\target(a)}-z_{\source(a)})a, \label{equation:psi}
  \end{align}
  for~$z,z_i\in\field$.
\end{alphatheorem}
After the statement of \cref{theorem:happel}
we explain how the sequence \eqref{equation:vector-fields-presentation}
is similar to a presentation of the first Hochschild cohomology of the path algebra~$\field Q$,
which is also a measure of a symmetry group,
leading to an isomorphism of vector spaces
\begin{equation}
  \HH^0(X,\tangent_X)\cong\hochschild^1(\field Q).
\end{equation}
This isomorphism has a precursor in
the relationship between (infinitesimal) symmetries of a variety
and (infinitesimal) symmetries of a moduli space of sheaves on the variety.
The first example is given by
the moduli space~$\moduli_C(r,\mathcal{L})$ of stable vector bundles
of rank~$r\geq 2$ and determinant~$\mathcal{L}$,
on the smooth projective curve~$C$ of genus~$g\geq 2$,
such that~$\gcd(r,\deg\mathcal{L})=1$,
for which there exists an isomorphism
\begin{equation}
  \label{equation:vector-fields-vbac}
  \HH^0(\moduli_C(r,\mathcal{L}),\tangent_{\moduli_C(r,\mathcal{L})})\cong\HH^0(C,\tangent_C),
\end{equation}
both sides being zero by \cite[Theorem~1(a)]{MR0384797}.
The second example is given by
the Hilbert scheme~$\operatorname{Hilb}^nS$ of~$n$ points on a smooth projective surface~$S$,
for which there exists an isomorphism
\begin{equation}
  \label{equation:vector-fields-hilbert-scheme}
  \HH^0(\operatorname{Hilb}^nS,\tangent_{\operatorname{Hilb}^nS})\cong\HH^0(S,\tangent_S)
\end{equation}
by \cite[Corollaire~1]{MR2932167}.

In \eqref{equation:vector-fields-hilbert-scheme},
the isomorphism is in fact induced from an inclusion of algebraic groups~$\Aut(S)\hookrightarrow\Aut(\operatorname{Hilb}^nS)$
(see more on this below the next statement),
which after taking Lie algebras means that \eqref{equation:vector-fields-hilbert-scheme}
is an isomorphism of Lie algebras.
This brings us to the following conjecture.

\begin{alphaconjecture}
  \label{conjecture:lie-algebra}
  In the setting of \cref{theorem:vector-fields}
  there exists a naturally induced isomorphism of Lie algebras
  \begin{equation}
    \label{equation:lie-algebra-isomorphism}
    \HH^0(X,\tangent_X)\cong\hochschild^1(\field Q),
  \end{equation}
  where the Lie algebra structure on the left (resp.~right)
  is given by the Schouten--Nijenhuis bracket of vector fields
  (resp.~the Gerstenhaber bracket).
\end{alphaconjecture}
In \cref{example:3-vertex} we give an example
where \eqref{equation:lie-algebra-isomorphism} is manifestly not an isomorphism of Lie algebras
without the ample stability condition,
following the failure of \eqref{equation:global-sections} being an isomorphism in \cref{theorem:global-sections}.

In fact, more precise results on the automorphism groups of these two families of moduli spaces are now available.
For~$\moduli_C(r,\mathcal{L})$ the automorphism group is described
in terms of automorphisms of~$C$ and~$r$-torsion in the Jacobian of~$C$ \cite{MR1336336}.
For~$\operatorname{Hilb}^nS$ the description of the automorphism group depends on the geometry of~$S$.
If~$S$ has a big and nef (anti)canonical bundle then~$\Aut(S)\cong\Aut(\operatorname{Hilb}^nS)$ \cite[Theorem~1]{MR4155174},
see also \cite[Theorem~1.3]{MR4117578} for a similar result for rational surfaces of Iitaka dimension at least one.
If however~$S$ is a K3 surface, a rich theory of \emph{non-natural} automorphisms exists, starting with \cite[\S6]{MR0728605}.

The extent to which \eqref{equation:lie-algebra-isomorphism} also holds
\emph{without} taking Lie algebras (and thus on the level of algebraic groups, where discrete contributions are possible),
and thus to which extent the analogue of \cite{MR4155174} holds,
is not clear.

\paragraph{Admissible embeddings}
The second application of \cref{theorem:global-sections}
should be seen in the context of Schofield's conjecture,
as stated on \cite[page~80]{MR1428456},
which says that~$\mathcal{U}$ is a partial tilting object.
It is settled in \cite{rigidity} for a large class of quiver moduli.
\Cref{theorem:global-sections} gives further information about this partial tilting object,
namely about the algebra structure on the (derived) endomorphisms.

We will rephrase this result using the Fourier--Mukai(-like) functor
\begin{equation}
  \label{equation:fourier-mukai}
  \Phi_{\mathcal{U}}\colon\derived^\bounded(\field Q)\to\derived^\bounded(\modulispace[\stable\theta]{Q,\mathbf{d}}),
\end{equation}
where we continue to assume \cref{assumption:standing-assumption},
so that there exists a universal representation~$\mathcal{U}$.
The functor \eqref{equation:fourier-mukai} is defined on objects as
\begin{equation}
  \Phi_{\mathcal{U}}(V)=\RsheafHom_{\mathcal{O}_XQ}(\mathcal{U},V\otimes_\field\mathcal{O}_X),
\end{equation}
where~$\sheafHom_{\mathcal{O}_XQ}(-,-)$ is the sheafy Hom for coherent left~$\mathcal{O}_XQ$-modules,
which has a natural coherent~$\mathcal{O}_X$-module structure in our setup,
see, e.g., \cite[\S3.1]{2307.01711v2}.
\begin{alphatheorem}
  \label{theorem:admissible}
  Let~$Q$, $\mathbf{d}$ and~$\theta$ be as in \cref{assumption:standing-assumption},
  and assume in addition that~$\mathbf{d}$ is~$\theta$-strongly amply stable.
  Consider~$X=\modulispace[\stable\theta]{Q,\mathbf{d}}$.
  The functor \eqref{equation:fourier-mukai}
  is fully faithful.
\end{alphatheorem}
There are in fact 4 natural Fourier--Mukai-like functors to be considered and compared.
We will discuss this in \cref{section:admissible},
and show that all~4~are fully faithful,
because they are all related to each other.

This result was only known in the thin (and thus toric) case
for the canonical stability condition,
by Altmann--Hille \cite[Theorem~1.3]{MR1688469},
and for quiver flag varieties,
by Craw--Ito--Karmazyn \cite[Example~2.9]{MR3803802}.
We state the precise conditions for the toric case in \cref{theorem:altmann-hille}
to illustrate how our methods and result generalize this setting.

As with \cref{theorem:vector-fields},
this result has parallel results in the context of moduli spaces of sheaves.
The first example is given by
moduli spaces of vector bundles on a curve,
for which the fully faithfulness of
\begin{equation}
  \Phi_{\mathcal{E}}\colon\derived^\bounded(C)\to\derived^\bounded(\moduli_C(r,\mathcal{L}))
\end{equation}
for the universal vector bundle~$\mathcal{E}$
on~$C\times\moduli_C(r,\mathcal{L})$
is in various levels of generality
obtained in \cite{MR3954042,MR3713871,MR3764066,MR4651618}.
The second example is given by
Hilbert schemes of points on surfaces,
for which the fully faithfulness of
\begin{equation}
  \label{equation:universal-ideal-sheaf-fully-faithful}
  \Phi_{\mathcal{I}}\colon\derived^\bounded(S)\to\derived^\bounded(\operatorname{Hilb}^nS)
\end{equation}
for the universal ideal sheaf~$\mathcal{I}$
on~$S\times\operatorname{Hilb}^nS$
if and only if~$\mathcal{O}_S$ is an exceptional object
is established in \cite{MR3397451,1909.04321v2}.

A link between admissible embeddings using universal objects
and vector fields (and deformation theory)
is explained for Hilbert schemes of points in \cite{MR3950704}.
In \cref{proposition:spectral-sequence} we will explain how the same method works for quiver moduli,
and thus how \cref{theorem:vector-fields}
and the rigidity result of \cite{rigidity} (recalled in \cref{corollary:rigidity})
can be obtained from the statement (and not the ingredients of the proof) of \cref{theorem:admissible}.

\paragraph{Acknowledgements}
P.B.~was partially supported by the Luxembourg National Research Fund (FNR--17113194). \\
H.F.~was partially supported by the Deutsche Forschungsgemeinschaft (DFG, German Research Foundation) -- SFB-TRR 358/1 2023 -- 491392403. \\
M.R.~was supported by the Deutsche Forschungsgemeinschaft (DFG, German Research Foundation) CRC-TRR~191 ``Symplectic structures in geometry, algebra and dynamics'' (281071066)

We want to thank Alastair Craw for interesting discussions regarding the case of quiver flag varieties.

\section{Quiver moduli}
\label{section:preliminaries}

\paragraph{Construction of the moduli space}
We first recall the GIT construction of the moduli space of stable quiver representations,
to set up the notation,
as introduced by King in \cite{MR1315461}.
For more background on this, one is referred to \cite{MR2484736},
and for a stacky construction one is referred to \cite{2210.00033v1}.

Let~$Q=(Q_0,Q_1)$ be a quiver,
where we denote~$\source(a)$ (resp.~$\target(a)$)
for the source (resp.~target) of~$a\in Q_1$.
We will denote the path algebra by~$\field Q$,
and throughout we will work with left~$\field Q$-modules.

Fixing a dimension vector~$\mathbf{d}\in\mathbb{N}^{Q_0}$
we have the affine space
\begin{equation}
  \Rep(Q,\mathbf{d})\colonequals\prod_{a\in Q_1}\mathbb{A}^{d_{\source(a)}d_{\target(a)}}
\end{equation}
as the parameter space for representations of~$Q$ of dimension vector~$\mathbf{d}$.
It comes with an action of the group
\begin{equation}
  \group{\mathbf{d}}\colonequals\prod_{i\in Q_0}\GL_{d_i}
\end{equation}
acting by conjugation in the usual way.
Its orbits are the isomorphism classes,
but to get a well-behaved moduli space we need to introduce one more ingredient.

Let~$\theta\in\Hom(\mathbb{Z}^{Q_0},\mathbb{Z})$ such that~$\theta(\mathbf{d})=0$,
which we call a \emph{stability parameter}.
Then we say that a representation~$M$
corresponding to a point~$M\in\Rep(Q,\mathbf{d})$ is~\emph{$\theta$-semistable} if~$\theta(\dimvect N)\leq 0$
for all non-zero and proper subrepresentations~$N$ of~$M$,
and we say it is~\emph{$\theta$-stable} if the inequality is strict.
We will tacitly identify~$\Hom(\mathbb{Z}^{Q_0},\mathbb{Z})$ and~$\mathbb{Z}^{Q_0}$ in what follows.

This gives us the~$\group{\mathbf{d}}$-stable open subsets
\begin{equation}
  \Rep^{\stable\theta}(Q,\mathbf{d})
  \subseteq
  \Rep^{\semistable\theta}(Q,\mathbf{d})
  \subseteq
  \Rep(Q,\mathbf{d}),
\end{equation}
which after the GIT quotient by~$\group{\mathbf{d}}$
with respect to the polarization given by~$\theta$
gives
\begin{equation}
  \modulispace[\stable\theta]{Q,\mathbf{d}}
  \subseteq
  \modulispace[\semistable\theta]{Q,\mathbf{d}}
  \to
  \modulispace{Q,\mathbf{d}}.
\end{equation}
We have that~$\modulispace[\stable\theta]{Q,\mathbf{d}}$
is a smooth variety,
the first morphism is an open immersion,
and the second is projective.
If~$Q$ is acyclic then~$\modulispace{Q,\mathbf{d}}\cong\operatorname{Spec}k$.

If $\mathbf{d}$ is indivisible then
we can obtain a universal representation~$\mathcal{U}=\mathcal{U}(\mathbf{a})$ on~$\modulispace[\stable\theta]{Q,\mathbf{d}}$,
which depends on the choice of an~$\mathbf{a}\in\Hom(\mathbb{Z}^{Q_0},\mathbb{Z})$
such that~$\mathbf{a}(\mathbf{d})=1$,
see, e.g., \cite[\S2.1]{2307.01711v2}.
We can decompose~$\mathcal{U}$ as~$\smash{\bigoplus_{i\in Q_0}\mathcal{U}_i}$,
such that at a point~$\smash{[M]\in\modulispace[\stable\theta]{Q,\mathbf{d}}}$
corresponding to an isomorphism class of~$\theta$-stable representations,
the fiber of~$\mathcal{U}_i$
is the vector space~$M_i$.

We can summarize the preceding setup as follows.
\begin{proposition}
  \label{proposition:nice}
  Let~$Q$ be an acyclic quiver,
  $\mathbf{d}$ a dimension vector,
  and~$\theta$ a stability parameter
  such that
  \begin{itemize}
    \item $\mathbf{d}$ is indivisible,
    \item every $\theta$-semistable representation of dimension vector $\mathbf{d}$ is $\theta$-stable.
  \end{itemize}
  Then~$\modulispace[\stable\theta]{Q,\mathbf{d}}$ is a smooth projective variety,
  which comes equipped with a universal bundle~$\mathcal{U}(\mathbf{a})$
  for every~$\mathbf{a}\in\Hom(\mathbb{Z}^{Q_0},\mathbb{Z})$ such that~$\mathbf{a}(\mathbf{d})=1$.
\end{proposition}

\paragraph{Ample stability}
We need that our quiver moduli spaces are particularly nice,
beyond what is assumed in \cref{proposition:nice}.
\begin{definition}
  \label{definition:amply-stable}
  A dimension vector~$\mathbf{d}$ is \emph{$\theta$-amply stable} if
  \begin{equation}
    \label{equation:amply-stable}
    \codim_{\Rep(Q,\mathbf{d})}(\Rep(Q,\mathbf{d})\setminus\Rep^{\stable\theta}(Q,\mathbf{d}))\geq 2.
  \end{equation}
\end{definition}
This condition in particular ensures that~$\Pic\modulispace[\stable\theta]{Q,\mathbf{d}}\cong\mathbb{Z}^{\#Q_0-1}$.
In \cref{example:3-vertex} we give an example where
\cref{theorem:vector-fields} fails,
when ample stability does not hold,
thus explaining why we really need to assume \cref{definition:amply-stable}.

We also need the following slightly stronger condition
in order to apply \cite{rigidity}.
As explained in op.~cit.,
it is expected that this condition can be omitted from the results in op.~cit.
\begin{definition}
  \label{definition:strongly-amply-stable}
  A dimension vector~$\mathbf{d}$ is \emph{$\theta$-strongly amply stable} if
  for every subdimension vector~$\mathbf{e}\leq\mathbf{d}$
  for which~$\mu(\mathbf{e})\geq\mu(\mathbf{d}-\mathbf{e})$
  we have~$\langle\mathbf{e},\mathbf{d}-\mathbf{e}\rangle\leq -2$.
\end{definition}
By \cite[Proposition~5.1]{MR3683503} we have that strong ample stability implies ample stability.
In \cite[Example~4.6]{rigidity}
an example is given where the converse implication does not hold.
For the proof of \cref{theorem:vector-fields} we will need the stronger notion,
because we will appeal to the main result of \cite{rigidity},
but as in op.~cit., we expect the results hold for ample stability.

\paragraph{Standing assumptions}
To ensure all the good properties discussed above,
we introduce the following conditions.
\begin{assumption}
  \label{assumption:standing-assumption}
  For $Q$, $\mathbf{d}$, and $\theta$ we assume that
  \begin{enumerate}[label=\textnormal{(\roman*)}]
    \item \label{item:acyclic} $Q$ is acyclic,
    \item \label{item:indivisible}$\mathbf{d}$ is indivisible,
    \item \label{item:semistable=stable} every $\theta$-semistable representation of dimension vector $\mathbf{d}$ is $\theta$-stable,
    \item \label{item:amply-stable}$\mathbf{d}$ is $\theta$-amply stable.
  \end{enumerate}
\end{assumption}

\section{Double framing construction}
\label{section:double-framing}
In this section, we will introduce a construction to reduce the computation of
sheaf cohomology of universal bundles on a quiver moduli space
to sheaf cohomology of universal line bundles on another quiver moduli space.
We will do this via a double framing construction
which will realize a fiber product of projectivizations of universal bundles
as a quiver moduli space.

\paragraph{On fiber products of projective bundles}
Let us first collect some general facts on projective bundles.
Let~$X$ be a variety
and let~$E$ be a vector bundle of rank~$r+1$ on~$X$.
Let
\begin{equation}
  p\colon\mathbb{P}(E)=\mathbb{P}_X(E)=\operatorname{Proj}\operatorname{Sym}_{\mathcal{O}_X}^\bullet(E)\to X
\end{equation}
be the projectivization of~$E$;
here~$\operatorname{Proj}$ denotes the relative Proj over~$X$.
Its fiber in a point~$x \in X$ consists of one-dimensional quotients of~$E_x$.
The universal line bundle~$\mathcal{O}_{\mathbb{P}(E)}(1)$ is a quotient of~$p^*E$.

Let~$F$ be another vector bundle on $X$,
and consider the projectivization
\begin{equation}
  q\colon\mathbb{P}(F)\to X.
\end{equation}
Define~$Y$ as the fiber product
\begin{equation}
  \label{equation:cartesian-product}
  \begin{tikzcd}
    Y \arrow{r}{q'} \arrow[swap]{d}{p'} & \mathbb{P}(E) \arrow{d}{p} \\
    \mathbb{P}(F) \arrow[swap]{r}{q} & X.
  \end{tikzcd}
\end{equation}
Then
\begin{equation}
  Y \cong \mathbb{P}_{\mathbb{P}(E)}(p^*F) \cong \mathbb{P}_{\mathbb{P}(F)}(q^*E).
\end{equation}

Let $m,n \in \mathbb{Z}$.
We define
\begin{equation}
  \mathcal{O}_Y(m,n)\colonequals q^{\prime*}\mathcal{O}_{\mathbb{P}(E)}(m) \otimes p^{\prime*}\mathcal{O}_{\mathbb{P}(F)}(n).
\end{equation}
Applying the projection formula and flat base change,
together with \cite[\href{https://stacks.math.columbia.edu/tag/01XX}{Lemma 01XX}]{stacks-project}
we have that
\begin{equation}
  \label{equation:cohomology-Y-vs-X}
  \HH^i(Y,\mathcal{O}_Y(1,1))
  \cong
  \HH^i(X,E\otimes F).
\end{equation}

\paragraph{Double framing construction}
Let $Q$ be a quiver,
$\mathbf{d}$ a dimension vector,
and $\theta$ a stability parameter,
satisfying \cref{assumption:standing-assumption}.

Let $X\colonequals\modulispace[\stable{\theta}]{Q,\mathbf{d}}$ denote the moduli space
of~$\theta$-(semi)stable representations of dimension vector~$\mathbf{d}$,
and let $\mathcal{U}$ be the universal representation on~$X$,
which depends on the choice of a character~$\mathbf{a}$ of weight one as in \cref{section:preliminaries}.

Fix vertices $i,j\in Q_0$; we allow~$i = j$.
Consider the fiber product
\begin{equation}
  Y = \mathbb{P}_X(\mathcal{U}_i^\vee) \times_X \mathbb{P}_X(\mathcal{U}_j).
\end{equation}
Our goal is to describe~$Y$ as a quiver moduli space.
Let~$\overline{Q}$ be the doubly-framed quiver defined by
\begin{equation}
  \left\{
  \begin{aligned}
    \overline{Q}_0 &\colonequals Q_0 \sqcup \{0,\infty\} \\
    \overline{Q}_1 &\colonequals Q_1 \sqcup \{0 \to i, j \to \infty\}.
  \end{aligned}
  \right.
\end{equation}
where~$0$ and~$\infty$ are two new symbols.
This quiver depends on the choice of $i$ and $j$,
but we will suppress this in the notation.
The notation~$\overline{Q}$ is not to be confused with the doubled quiver,
as for instance in the definition of the preprojective algebra.

We define a doubly-framed dimension vector $\overline{\mathbf{d}}$ by
\begin{equation}
  \label{equation:double-framed-d}
  \overline{d}_k
  \colonequals
  \begin{cases}
    d_k & k\in Q_0 \\
    1 & k\in\{0,\infty\}.
  \end{cases}
\end{equation}
We will also denote this dimension vector as $\overline{\mathbf{d}} = (1,\mathbf{d},1)$.

A representation of $\overline{Q}$ of dimension vector $\overline{\mathbf{d}}$ is
a triple $(v,M,\phi)$ which consists of
\begin{itemize}
  \item a representation $M$ of~$Q$ with dimension vector $\mathbf{d}$,
  \item a vector $v \in M_i$,
  \item a linear form $\phi \in M_j^\vee$.
\end{itemize}
The representation variety $\representationvariety{\overline{Q},\overline{\mathbf{d}}}$ can therefore be written as
\begin{equation}
  \representationvariety{\overline{Q},\overline{\mathbf{d}}} = \mathbb{A}^{d_i} \times \representationvariety{Q,\mathbf{d}} \times \mathbb{A}^{d_j},
\end{equation}
where the factor~$\mathbb{A}^{d_i}$ encodes the choice of~$v\in M_i$,
and the factor~$\mathbb{A}^{d_j}$ encodes the choice of~$\phi\in M_j^\vee$.

Next, we fix a natural number $N$,
which we will choose sufficiently big later on.
We define a stability parameter $\overline{\theta}=\overline{\theta}_N$ as~$(1,N\theta,-1)$
using the identification~$\mathbb{Z}^{\overline{Q}_0}=\mathbb{Z}\oplus\mathbb{Z}^{Q_0}\oplus\mathbb{Z}$
where the first (resp.~third) summand
corresponds to~$k=0$ (resp.~$k=\infty$).
The following proposition describes the stable and semistable locus with respect to $\overline{\theta}$.

\begin{proposition}
  \label{proposition:semistability-double-framed}
  Let $(v,M,\phi)$ be a representation of $\overline{Q}$ of dimension vector $\overline{\mathbf{d}}$. The following are equivalent:
  \begin{enumerate}[label=\textnormal{(\roman*)}]
    \item \label{item:stable-double-framed} $(v,M,\phi)$ is $\overline{\theta}$-stable.
    \item \label{item:semistable-double-framed} $(v,M,\phi)$ is $\overline{\theta}$-semistable.
    \item \label{item:description-semistability} $M$ is $\theta$-(semi)stable, $v \neq 0$, and $\phi \neq 0$.
  \end{enumerate}
\end{proposition}

We first prove a lemma.
For a stability parameter $\theta$ for $Q$ such that $\theta(\mathbf{d}) = 0$ we define the following sets of dimension vectors
\begin{align}
  \label{equation:set-B-plus}
  B_+(\theta) &= \{\mathbf{e} \in \mathbb{N}_0^{Q_0} \mid 0 \leq \mathbf{e} \leq \mathbf{d} \text{ and } \theta(\mathbf{e}) > 0\} \\
  \label{equation:set-B-minus}
  B_-(\theta) &= \{\mathbf{e} \in \mathbb{N}_0^{Q_0} \mid 0 \leq \mathbf{e} \leq \mathbf{d} \text{ and } \theta(\mathbf{e}) < 0\} \\
  \label{equation:set-B-zero}
  B_0(\theta) &= \{\mathbf{e} \in \mathbb{N}_0^{Q_0} \mid 0 \leq \mathbf{e} \leq \mathbf{d} \text{ and } \theta(\mathbf{e}) = 0\}
\end{align}
They depend on $\mathbf{d}$ but as the dimension vector will be clear from the context,
we choose to neglect the dependency on $\mathbf{d}$ in the notation.
We determine the analogous sets for $\overline{\theta}$
(with respect to $\overline{\mathbf{d}}=(1,\mathbf{d},1)$ as in \eqref{equation:double-framed-d}).
\begin{lemma}
  \label{lemma:new-B}
  For~$N$ sufficiently large we have
  \begin{equation}
    \label{equation:new-B}
    \begin{aligned}
      B_+(\overline{\theta})
      &=\{ (1,\mathbf{e},0) \mid \mathbf{e} \in B_0(\theta)\} \cup \{ (0,\mathbf{e},0), (1,\mathbf{e},0), (0,\mathbf{e},1), (1,\mathbf{e},1) \mid \mathbf{e} \in B_+(\theta)\} \\
      B_-(\overline{\theta})
      &=\{ (0,\mathbf{e},1) \mid \mathbf{e} \in B_0(\theta)\} \cup \{ (0,\mathbf{e},0), (1,\mathbf{e},0), (0,\mathbf{e},1), (1,\mathbf{e},1) \mid \mathbf{e} \in B_-(\theta)\} \\
      B_0(\overline{\theta})
      &=\{ (0,\mathbf{e},0), (1,\mathbf{e},1) \mid \mathbf{e} \in B_0(\theta)\}.
    \end{aligned}
  \end{equation}
\end{lemma}

\begin{proof}
  Let $\mathbf{e}$ be a dimension vector such that $0 \leq \mathbf{e} \leq \mathbf{d}$.
  For any dimension vector of the form $(a,\mathbf{e},b)$ with~$a,b \in \{0,1\}$ we have
  \begin{equation}
    \overline{\theta}(a,\mathbf{e},b) = a + N\theta(\mathbf{e}) - b.
  \end{equation}
  The equalities in \eqref{equation:new-B} are then immediate,
  using that~$\overline{\theta}(a,\mathbf{e},b)>0$ for~$N\gg0$
  if~$\mathbf{e}\in B_+(\theta)$,
  respectively~$\overline{\theta}(a,\mathbf{e},b)<0$ for~$N\gg0$
  if~$\mathbf{e}\in B_-(\theta)$,
  and $\overline{\theta}(a,\mathbf{e},b) = a-b$ if $\mathbf{e} \in B_0(\theta)$.
\end{proof}

\begin{proof}[Proof of \cref{proposition:semistability-double-framed}]
  \cref{item:stable-double-framed} obviously implies \cref{item:semistable-double-framed}.

  Now we prove that \cref{item:semistable-double-framed} implies \cref{item:description-semistability}.
  Let $(v,M,\phi)$ be a $\overline{\theta}$-semistable representation.
  Let $\mathbf{e} \in B_+(\theta)$.
  As $(0,\mathbf{e},1)$ lies in $B_+(\overline{\theta})$,
  we see that $M$ cannot have a subrepresentation $M'$ of dimension vector $\mathbf{e}$,
  for otherwise $(0,M',\phi|_{M'_j})$ would be a subrepresentation of $(v,M,\phi)$ of dimension vector $(0,\mathbf{e},1)$
  contradicting semistability of $(v,M,\phi)$.
  Also, $v \neq 0$ because $(1,\mathbf{0},0) \in B_+(\overline{\theta})$
  and $\phi \neq 0$ because $(1,\mathbf{d},0) \in B_+(\overline{\theta})$.

  Finally, we show that \cref{item:description-semistability} implies \cref{item:stable-double-framed}.
  Let $(a,\mathbf{e},b) \in B_+(\overline{\theta}) \cup B_0(\overline{\theta})$.
  By the description of the sets $B_+(\overline{\theta})$ and $B_0(\overline{\theta})$ in \cref{lemma:new-B},
  we see that $\theta(\mathbf{e}) \geq 0$.
  Assume that $(v,M,\phi)$ had a subrepresentation of dimension vector $(a,\mathbf{e},b)$.
  In particular, $M$ would have a subrepresentation of dimension vector $\mathbf{e}$ which,
  by semistability,
  implies that $\theta(\mathbf{e}) \leq 0$.
  So $\theta(\mathbf{e}) = 0$ must hold.
  As $M$ is $\theta$-stable,
  this implies that $\mathbf{e} \in \{\mathbf{0},\mathbf{d}\}$.
  This shows that $(a,\mathbf{e},b)$ is one of the following six dimension vectors:
  \begin{equation}
    (0,\mathbf{0},0),\ (1,\mathbf{0},0),\ (1,\mathbf{0},1),\ (0,\mathbf{d},0),\ (1,\mathbf{d},0),\ (1,\mathbf{d},1)
  \end{equation}
  There can be no subrepresentations of $(v,M,\phi)$ of dimension vector $(1,\mathbf{0},0)$ or $(1,\mathbf{0},1)$
  because $v \neq 0$.
  There can also be no subrepresentations of dimension vectors $(0,\mathbf{d},0)$ or $(1,\mathbf{d},0)$,
  because $\phi\neq 0$.
  The remaining two possibilities are the zero dimension vector and $\overline{\mathbf{d}}$.
  We have established that $(v,M,\phi)$ is~$\overline{\theta}$-stable.
\end{proof}

\begin{remark}
  The double framing construction resembles the framing construction in \cite[Definition~3.1]{MR2511752}
  for the construction of smooth models.
  Given a triple $(Q,\mathbf{d},\theta)$ consisting of a quiver, a dimension vector, and a stability parameter,
  loc.~cit.~yields, when choosing~$\mathbf{n} = \mathbf{e}_i$
  another such triple $(\hat{Q},\hat{\mathbf{d}},\hat{\theta})$
  for which a $\hat{\theta}$-(semi)stable representation is a pair $(v,M)$
  such that $M$ is a $\theta$-semistable representation of~$Q$ and~$v \in M_i\setminus\{0\}$.

  Dualizing the construction of loc.~cit.~we obtain
  for a triple $(Q,\mathbf{d},\theta)$ as above
  a triple $(\tilde{Q},\tilde{\mathbf{d}},\tilde{\theta})$
  for which a $\tilde{\theta}$-(semi)stable representation is a pair $(M,\phi)$
  such that~$M$ is a $\theta$-semistable representation of~$Q$ and~$\phi\colon M_j \to \field$ is a non-zero linear form.

  Applying both constructions yields,
  independently of the order,
  the doubly-framed quiver $\overline{Q}$ and the dimension vector $\overline{\mathbf{d}}$.
  For the stability parameter though, the order matters.
  The two resulting stability parameters do not agree, i.e.,
  \begin{equation}
    \tilde{\hat{\theta}} \neq \hat{\tilde{\theta}}
  \end{equation}
  and they are both different from $\overline{\theta}$.
  One can a posteriori use \cref{proposition:semistability-double-framed}
  to conclude that they are all GIT-equivalent,
  meaning that the semistable loci with respect to these stability parameters agree.
\end{remark}

\begin{remark}
  Note also that the present double framing construction is essentially different from the one in \cite{MR4439387},
  used for modelling neural network architectures using quiver moduli.
  Namely, the double framing in op.~cit.~adds a single vertex to the quiver,
  as well as arrows to and from it, in contrast to the two different extension vertices used here
\end{remark}

\cref{proposition:semistability-double-framed} implies the following.
\begin{corollary}
  \label{corollary:doubly-framed-standing-assumption}
  The doubly-framed $\overline{Q}$, $\overline{\mathbf{d}}$, and $\overline{\theta}$
  satisfy \ref{item:acyclic}, \ref{item:indivisible}, and \ref{item:semistable=stable} of \cref{assumption:standing-assumption}.
\end{corollary}

Therefore $\modulispace[\stable{\overline{\theta}}]{\overline{Q},\overline{\mathbf{d}}}$ is smooth and projective,
and it possesses a universal representation $\overline{\mathcal{U}}=\overline{\mathcal{U}}(\overline{\mathbf{a}})$
dependent on the choice of an~$\overline{\mathbf{a}}$ such that~$\overline{\mathbf{a}}\cdot\overline{\mathbf{d}}=1$.
The summands~$\overline{\mathcal{U}}_0$ and~$\overline{\mathcal{U}}_\infty$ are line bundles.

The following lemma explains why we might still have to modify~$\overline{Q}$,
so that we can guarantee \ref{item:amply-stable} of \cref{assumption:standing-assumption}.

\begin{lemma}
  \label{lemma:ample-stability-generic-case}
  The dimension vector $\overline{\mathbf{d}}$ is amply stable for $\overline{\theta}$ if and only if $d_i > 1$ and $d_j > 1$.
\end{lemma}

\begin{proof}
  By \cref{proposition:semistability-double-framed}, the $\overline{\theta}$-unstable locus is the union
  \begin{equation}
    \representationvariety{\overline{Q},\overline{\mathbf{d}}} \setminus \representationvariety[\semistable{\overline{\theta}}]{\overline{Q},\overline{\mathbf{d}}} = \{ (v,M,\phi) \mid M \text{ is $\overline{\theta}$-unstable}\} \cup \{(v,M,\phi) \mid v = 0\} \cup \{(v,M,\phi) \mid \phi = 0\}.
  \end{equation}
  The first set has codimension at least 2 by $\theta$-ample stability of $\mathbf{d}$. The second set has codimension $d_i$ and the third has codimension $d_j$. This proves the claimed equivalence.
\end{proof}

We have the forgetful morphism
\begin{equation}
  u\colon\modulispace[\stable{\overline{\theta}}]{\overline{Q},\overline{\mathbf{d}}} \to \modulispace[\stable{\theta}]{Q,\mathbf{d}}:
  [v,M,\phi]\mapsto[M],
\end{equation}
which forgets the framing data.
We obtain $u^*\mathcal{U}_k \cong \overline{\mathcal{U}}_k$ for~$k\in Q_0$,
by choosing the character~$\overline{\mathbf{a}}$ as~$(0,\mathbf{a},0)$.

Using \cref{proposition:semistability-double-framed} we are now going to show the two spaces
\begin{align}
  \label{equation:overline-X}
  \overline{X}&\colonequals\modulispace[\stable{\overline{\theta}}]{\overline{Q},\overline{\mathbf{d}}} \\
  Y &\colonequals\mathbb{P}(\mathcal{U}_i^\vee) \times_X \mathbb{P}(\mathcal{U}_j)
\end{align}
are isomorphic. To prove this, let $f\colon Y \to X$ be the diagonal morphism in the diagram \eqref{equation:cartesian-product}.
\begin{proposition}
  \label{proposition:doubly-framed-as-quiver-moduli}
  There exists an isomorphism $Y\cong \overline{X}$ over $X$
  such that
  \begin{itemize}
    \item $\overline{\mathcal{U}}_k$ is identified with~$f^*\mathcal{U}_k$ for all $k \in Q_0$
      and $\overline{\mathcal{U}}_a$ is identified with~$f^*\mathcal{U}_a$ for all $a \in Q_1$,
    \item $\overline{\mathcal{U}}_0$ is identified with~$\mathcal{O}_Y(-1,0)$
      and $\overline{\mathcal{U}}_{0 \to i}$ is identified with $\mathcal{O}_Y(-1,0) \to f^*\mathcal{U}_i$,
    \item $\overline{\mathcal{U}}_\infty$ is identified with~$\mathcal{O}_Y(0,1)$
      and $\overline{\mathcal{U}}_{j \to \infty}$ is identified with $f^*\mathcal{U}_j \to \mathcal{O}_Y(0,1)$.
  \end{itemize}
\end{proposition}

\begin{proof}
  We obtain morphisms in both directions from the universal properties of~$Y$ and~$\overline{X}$ as follows.

  Let~$(v,M,\phi)$ be a~$\overline{\theta}$-stable doubly-framed representation.
  Then~$v \neq 0$ and~$\phi \neq 0$ by \cref{proposition:semistability-double-framed}.
  This implies that~$\overline{\mathcal{U}}_0$ is a line subbundle of~$\overline{\mathcal{U}}_i = u^*\mathcal{U}_i$
  and~$\overline{\mathcal{U}}_\infty$ is a line bundle quotient of~$\overline{\mathcal{U}}_j = u^*\mathcal{U}_j$.
  The universal property of~$Y$ thus yields a morphism~$\overline{X} \to Y$ over $X$
  under which~$f^*\mathcal{U}_i^\vee \to \mathcal{O}_Y(1,0)$
  pulls back to~$\smash{\overline{\mathcal{U}}}_{0 \to i}^\vee\colon \smash{\overline{\mathcal{U}}}_i^\vee \to \smash{\overline{\mathcal{U}}}_0^\vee$
  and~$f^*\mathcal{U}_j \to \mathcal{O}_Y(0,1)$
  pulls back to~$\smash{\overline{\mathcal{U}}_{j \to \infty}}\colon \smash{\overline{\mathcal{U}}_j} \to \smash{\overline{\mathcal{U}}}_\infty$.

  Conversely,
  consider the pullback~$f^*\mathcal{U}$
  of the universal representation on~$X = \modulispace[\stable{\theta}]{Q,\mathbf{d}}$.
  Using
  the morphisms~$\mathcal{O}_Y(0,1)^\vee \to f^*\mathcal{U}_i$
  and~$\mathcal{O}_Y(1,0)$ and~$\mathcal{U}_j \to \mathcal{O}_Y(1,0)$,
  we obtain a representation of~$\overline{Q}$ over~$Y$ of rank vector~$\smash{\overline{\mathbf{d}}}$
  such that the fiber over every point in~$Y$ is~$\smash{\overline{\theta}}$-stable;
  the last statement follows again from \cref{proposition:semistability-double-framed}.
  The universal property of~$\smash{\overline{X}}$ yields a morphism~$\smash{Y \to \overline{X}}$
  under which the universal representation~$\smash{\overline{\mathcal{U}}}$
  pulls back to the representation over~$Y$ described above.

  The two morphisms can be easily seen to be inverse bijections on closed points,
  which suffices to show that~$\overline{X}$ and~$Y$ are isomorphic.
  The identification of the universal bundles follows from the universal properties
  used in the construction.
  This proves the proposition.
\end{proof}

The previous proposition implies, using \eqref{equation:cohomology-Y-vs-X}:

\begin{corollary}
  \label{corollary:double-framed-global-sections}
  The diagram
  \begin{equation}
    \label{equation:diagram-double-framed-global-sections}
    \begin{tikzcd}
      p \arrow[mapsto]{d}{} &[-2em] e_j\field Qe_i \arrow{r}{h_{i,j}^\mathcal{U}} \arrow{d}{} & \HH^0(X,\mathcal{U}_i^\vee \otimes \mathcal{U}_j) \arrow{d}{} &[-2em] s \arrow[mapsto]{d}{} \\
      (j\to\infty)p(0\to i)& e_\infty \field\overline{Q}e_0 \arrow{r}{h_{0,\infty}^{\smash{\overline{\mathcal{U}}}}} & \HH^0(\overline{X},\smash{\overline{\mathcal{U}}}_0^\vee \otimes \smash{\overline{\mathcal{U}}}_\infty) & \smash{\overline{\mathcal{U}}}_{j \to \infty} \circ u^*s \circ \smash{\overline{\mathcal{U}}}_{0 \to i}
    \end{tikzcd}
  \end{equation}
  is commutative and the vertical maps are isomorphisms.
\end{corollary}

Next, let us give a description of $\overline{X}$ from \eqref{equation:overline-X}
as a quiver moduli space,
which in addition satisfies condition \ref{item:amply-stable} of \cref{assumption:standing-assumption}.

\begin{proposition}
  \label{proposition:annoying-special-cases-ample-stability}
  Let~$Q$, $\mathbf{d}$ and~$\theta$ be as in \cref{assumption:standing-assumption},
  and let~$\overline{Q}$, $\overline{\mathbf{d}}$ and~$\overline{\theta}$ be as constructed above.
  There exists a full subquiver $Q'$ of $\overline{Q}$ with $Q_0 \subseteq Q'_0$ and two vertices $i',j' \in Q'_0$ such that:
  \begin{enumerate}[label=\textnormal{(\roman*)}]
    \item \label{item:bijection-paths} There exist paths $q_0\colon 0 \to i'$ and $q_\infty\colon j' \to \infty$
      of lengths at most 1 such that
      \begin{equation}
        e_{j'}\field Q' e_{i'} \to e_\infty \field\overline{Q} e_0,\ p \mapsto q_\infty pq_0
      \end{equation}
      is an isomorphism.
    \item Let $\mathbf{d}' = \smash{\overline{\mathbf{d}}|_{Q'_0}}$ and consider the forgetful map
      \begin{equation}
        \representationvariety{\overline{Q},\overline{d}} \to \representationvariety{Q',\mathbf{d}'},\ N \mapsto N|_{Q'}.
      \end{equation}
      There exists a stability parameter $\theta'$ with $\theta'\cdot \mathbf{d}' = 0$
      such that for every representation $N$ of $\overline{Q}$ of dimension vector $\overline{\mathbf{d}}$
      which is $\overline{\theta}$-(semi)stable the representation $N|_{Q'}$ is $\theta'$-(semi)stable.
    \item For
      \begin{equation}
        \label{equation:X-prime}
        X' \colonequals \modulispace[\stable{\theta'}]{Q',\mathbf{d}'}
      \end{equation}
      with universal bundle~$\mathcal{U}'$
      the map $g\colon \overline{X} \to X'$ induced by $N \mapsto N|_{Q'}$
      is an isomorphism such that $\overline{\mathcal{U}}_{q_0}\colon \overline{\mathcal{U}}_0 \to \overline{\mathcal{U}}_{i'} = g^*\mathcal{U}'_{i'}$
      and~$\overline{\mathcal{U}}_{q_\infty}\colon g^*\mathcal{U}'_{j'} = \overline{\mathcal{U}} \to \overline{\mathcal{U}}_\infty$
      are isomorphisms.
    \item The data $Q'$, $\mathbf{d}'$, and $\theta'$ fulfill \cref{assumption:standing-assumption}, and $d'_{i'} = d'_{j'} = 1$
  \end{enumerate}
\end{proposition}

\begin{proof}
  We need to distinguish between different scenarios.
  \begin{enumerate}[label=(\alph*)]
    \item
      If $d_i > 1$ and $d_j > 1$, then we set $Q'\colonequals\overline{Q}$, $i'\colonequals0$, and~$j'\colonequals\infty$.
      \Cref{item:bijection-paths} is fulfilled with $q_0 = e_0$ and $q_\infty = e_\infty$.
      We get $\mathbf{d}'=\overline{\mathbf{d}}$ and take $\theta'\colonequals\overline{\theta}$.
      The map $g$ is the identity.
      \Cref{lemma:ample-stability-generic-case} tells us that~%
      $\mathbf{d}'$ is amply stable for $\theta'$ in this case.
      Together with \cref{corollary:doubly-framed-standing-assumption},
      we now deduce that \cref{assumption:standing-assumption} is satisfied.

    \item \label{item:single-framed-1}
      If $d_i > 1$ and $d_j = 1$,
      then let $Q'$ be the full subquiver with $Q'_0 = \overline{Q}_0 \setminus \{\infty\} = Q_0 \sqcup \{0\}$ which is acyclic. Let $i'=0$ and $j'=j$. Then with $q_0 = e_0$ and $q_\infty = (j \to \infty)$, \cref{item:bijection-paths} is satisfied.

      The dimension vector $\mathbf{d}' = (1,\mathbf{d})$ is indivisible, so it meets \cref{assumption:standing-assumption}\ref{item:indivisible}.
      Let~$\theta'$ be given by~$(\left|\mathbf{d}\right|,(\left|\mathbf{d}\right|+1)N\theta-1,0)$ for~$N\gg0$.
      Here, $\left|\mathbf{d}\right| = \sum_{i \in Q_0} d_i$.

      A representation of $Q'$ of dimension vector~$\mathbf{d}'$ is
      a pair~$(v,M)$ consisting of a representation~$M$ of~$Q$ of dimension vector~$\mathbf{d}$
      and a vector~$v \in M_i$.
      In \cite[Proposition~3.3]{MR2511752} it is shown that the following are equivalent:
      \begin{itemize}
        \item $(v,M)$ is $\theta'$-semistable
        \item $(v,M)$ is $\theta'$-stable
        \item $M$ is $\theta$-semistable and $v \neq 0$.
      \end{itemize}
      This shows that \cref{assumption:standing-assumption}\ref{item:semistable=stable} holds,
      and also, with the same proof as \cref{lemma:ample-stability-generic-case},
      that \cref{assumption:standing-assumption}\ref{item:amply-stable} is satisfied.

      The forgetful map is
      \begin{equation}
        \representationvariety{\overline{Q},\overline{d}} \to \representationvariety{Q',\mathbf{d}'},\ (v,M,\phi) \mapsto (v,M).
      \end{equation}
      A representation $(v,M,\phi)$ of dimension vector $\overline{\mathbf{d}}$ is $\overline{\theta}$-semistable
      if and only if $\phi$ is an isomorphism and $(v,M)$ is $\theta'$-stable.
      As we can use the $\group{\overline{\mathbf{d}}}$-action to transform $\phi$ to the identity,
      we obtain an isomorphism
      \begin{equation}
        g\colon\modulispace[\stable{\overline{\theta}}]{\overline{Q},\overline{\mathbf{d}}} \to \modulispace[\stable{\theta'}]{Q',\mathbf{d}'}
      \end{equation}
      which pulls back $\mathcal{U}'_{i'}$ to $\overline{\mathcal{U}}_0$
      and $\mathcal{U}'_{j'}$ to $\overline{\mathcal{U}}_j \cong \overline{\mathcal{U}}_\infty$.

    \item If $d_i = 1$ and $d_j > 1$,
      then we define $Q'$ as the full subquiver with $Q'_0 = \overline{Q}_0 \setminus \{0\} = Q_0 \sqcup \{\infty\}$.
      Let $i'=i$, $j'=\infty$, and $\theta'$ be defined by $((\left|\mathbf{d}\right|+1)N\theta+1,-\left|\mathbf{d}\right|)$ for~$N\gg0$.
      We may argue dually to the proof of \cref{item:single-framed-1}.

    \item If $d_i = d_j = 1$, then we let $Q' = Q$, $i'=i$, and $j'=j$. With $\theta' = \theta$ the claim is then obviously true.
      \qedhere
  \end{enumerate}
\end{proof}

\begin{corollary}
  \label{corollary:amply-stable-global-sections}
  The diagram
  \begin{equation}
    \label{equation:diagram-amply-stable-global-sections}
    \begin{tikzcd}
      p \arrow[mapsto]{d}{} &[-2em] e_{j'}\field Q'e_{i'} \arrow{r}{h_{i',j'}^{\mathcal{U}'}} \arrow{d}{} & \HH^0(X,{\mathcal{U}_{i'}}^{\prime\vee} \otimes \mathcal{U}'_{j'}) \arrow{d}{} &[-2em] s \arrow[mapsto]{d}{} \\
      q_\infty pq_0& e_\infty \field\overline{Q}e_0 \arrow{r}{h_{0,\infty}^{\smash{\overline{\mathcal{U}}}}} & \HH^0(\overline{X},\smash{\overline{\mathcal{U}}}_0^\vee \otimes \smash{\overline{\mathcal{U}}}_\infty) & \smash{\overline{\mathcal{U}}}_{q_\infty} \circ g^*s \circ \smash{\overline{\mathcal{U}}}_{q_0}
    \end{tikzcd}
  \end{equation}
  is commutative and the vertical maps are isomorphisms.
\end{corollary}

In the next section we will compute the global sections of the \emph{line bundle}~$\mathcal{U}_{i'}^\vee\otimes\mathcal{U}_{j'}$.

\begin{remark}
  The double framing construction performed here,
  in order to obtain the reduction of the cohomology of vector bundles
  to that of line bundles as in \cref{corollary:double-framed-global-sections}
  is parallel to what is done in \cite[Proof of Theorem~7.1]{2201.10033v2}
  for computing the cohomology of~$\mathcal{U}_{x_1}^\vee\otimes\mathcal{U}_{x_2}$
  for two distinct closed points~$x_1,x_2\in C$
  on the moduli space~$\moduli_C(r,\mathcal{L})$ of vector bundles on a curve.

  Namely, denote~$\mathcal{U}_x\colonequals\mathcal{U}|_{\{x\}\times\moduli_C(r,\mathcal{L})}$
  where~$\mathcal{U}$ is the universal vector bundle on~$C\times\moduli_C(r,\mathcal{L})$.
  In op.~cit.~the double framing is performed by considering the moduli space~$\moduli_C(r,\mathcal{L},\mathbf{e})$
  of parabolic bundles,
  where the parabolic structure is considered in the two points~$x_1$ and~$x_2$,
  leading to the isomorphism
  \begin{equation}
    \HH^i(\moduli_C(r,\mathcal{L}),\mathcal{U}_{x_1}^\vee\otimes\mathcal{U}_{x_2})
    \cong
    \HH^i(\moduli_C(r,\mathcal{L},\mathbf{e}),\mathcal{O}(1,1)).
  \end{equation}
  One difference
  is that in op.~cit.~wall-crossing methods are used
  to show vanishing of all cohomology,
  whereas that is impossible in our setup,
  as we can (and will) have global sections.
\end{remark}

\section{Global sections}
\subsection{Global sections of the endomorphism bundle}
The universal representation~$\mathcal{U}$ depends on the choice of the normalization~$\mathbf{a}$.
But this dependence disappears when considering~$\sheafHom(\mathcal{U},\mathcal{U})$.
We will thus compute the global sections of its summands~$\mathcal{U}_i^\vee\otimes\mathcal{U}_j$.
The higher cohomology will be dealt with using \cref{theorem:cohomology-vanishing},
which is taken from \cite{rigidity}.

The proof of \cref{theorem:global-sections} uses the reduction techniques from \cref{section:double-framing}
and a lemma which is a consequence of the celebrated result of Le Bruyn and Procesi \cite[Theorem~1]{MR0958897},
whose statement we first recall.

\begin{theorem}[Le Bruyn--Procesi]
  \label{theorem:le-bruyn-procesi}
  Let $Q$ be a (not necessarily acyclic) quiver and let $\mathbf{d}$ be a dimension vector.
  The ring $\field[\representationvariety{Q,\mathbf{d}}]^{\group{\mathbf{d}}}$ of invariant functions is generated,
  as a $\field$-algebra,
  by the functions
  \begin{equation}
    \varphi_c\colon\representationvariety{Q,\mathbf{d}} \to \field,\ M \mapsto \operatorname{tr}(M_c)
  \end{equation}
  where $c$ ranges over all oriented cycles of $Q$.
\end{theorem}

Let~$p = a_\ell\cdots a_1$ be an oriented path in the quiver~$Q$.
Let~$\mathbf{d}$ be a dimension vector for which we have~$d_{\source(p)}=d_{\target(p)}=1$.
As every $N \in \representationvariety{Q,\mathbf{d}}$ comes equipped with a basis
for each of the vector spaces $N_i$,
we obtain $\Hom_\field(N_{\source(p)},N_{\target(p)}) \cong \field$.
We define a regular function
\begin{equation}
  f_p\colon\representationvariety{Q,\mathbf{d}} \to \Hom_\field(N_{\source(p)},N_{\target(p)}) \cong \field,\ N\mapsto N_p;
\end{equation}
where we identify the linear map $N_p = N_{a_\ell} \circ \ldots \circ N_{a_1}$ with the factor of the scalar multiplication which it performs on the basis vectors.
As
\[
  f_p(g\cdot N) = g_{\target(p)}g_{\source(p)}^{-1}N_p
\]
for every $N \in \representationvariety{Q,\mathbf{d}}$ and every $g \in \group{\mathbf{d}}$, we see that $f_p$ is a semi-invariant function of weight $\delta_{\target(p)} - \delta_{\source(p)}$.

This allows us to prove the following.
\begin{lemma}
  \label{lemma:semi-invariant-functions}
  Let $Q$ be an acyclic quiver.
  Let $0, \infty \in Q_0$ be two vertices,
  and let $\mathbf{d}$ be a dimension vector such that $d_0 = d_\infty = 1$.
  Then the morphism
  \begin{equation}
    e_\infty\field Q e_0 \to \field[\representationvariety{Q,\mathbf{d}}]^{\group{\mathbf{d}},\delta_\infty-\delta_0}:p \mapsto f_p
  \end{equation}
  is an isomorphism of $\field$-vector spaces.
\end{lemma}

\begin{proof}
  If $0 = \infty$ then the theorem of Le Bruyn and Procesi directly applies -- note that the quiver is assumed to be acyclic.
  We may therefore assume that $0 \neq \infty$.

  We consider the subgroup
  \begin{equation}
    H = \{ g = (g_i)_{i \in Q_0} \in \group{\mathbf{d}} \mid g_0 = g_\infty \}
  \end{equation}
  of $\group{\mathbf{d}}$.
  There is an isomorphism $\psi\colon H \times \Gm \to \group{\mathbf{d}}$
  of algebraic groups defined by $\psi(h,t) = g$, where
  \begin{equation}
    g_i = \begin{cases} h_i & \text{if } i \neq \infty \\ th_\infty & \text{if } i = \infty. \end{cases}
  \end{equation}
  Now consider the algebra $\field[\representationvariety{Q,\mathbf{d}}]^H$ of $H$-invariant polynomial functions.
  It is a $\mathbb{Z}$-graded algebra by the action of $\Gm$.
  The space of semi-invariant functions with respect to $\group{\mathbf{d}}$ of weight $\delta_\infty - \delta_0$
  identifies, via $\psi$,
  with the subspace of $\field[\representationvariety{Q,\mathbf{d}}]^H$ on which $\Gm$ acts linearly, i.e.,
  \begin{equation}
    \field[\representationvariety{Q,\mathbf{d}}]^{\group{\mathbf{d}},\delta_\infty - \delta_0}
    =
    \left( \field[\representationvariety{Q,\mathbf{d}}]^H \right)_1.
  \end{equation}
  We will identify the latter with the invariant polynomial functions on a representation variety of a different quiver.
  Namely, let $Q^\flat$ be the quiver which arises from $Q$ by identifying the vertices $0$ and $\infty$;
  denote the ``merged'' vertex by $0\infty$.
  Let $\mathbf{d}^\flat$ be the dimension vector with $d^\flat_i = d_i$ for all $i \in Q^\flat_0 \setminus \{0\infty\}$ and $d^\flat_{0\infty} = 1$.
  There exist isomorphisms
  \begin{equation}
    \representationvariety{Q^\flat,\mathbf{d}^\flat} \cong \representationvariety{Q,\mathbf{d}}
  \end{equation}
  and
  \begin{equation}
    \group{\mathbf{d}^\flat} \cong H.
  \end{equation}
  Under these identifications, the action of $\group{\mathbf{d}^\flat}$ agrees with the action of $H$.
  Now we may apply \cref{theorem:le-bruyn-procesi}.
  It tells us that $\field[\representationvariety{Q^\flat,\mathbf{d}^\flat}]^{\group{\mathbf{d}^\flat}} \cong \field[\representationvariety{Q,\mathbf{d}}]^H$
  is generated by traces along oriented cycles in $Q^\flat$.
  As $Q$ is acyclic, all oriented cycles in $Q^\flat$ must involve $0\infty$.

  For~$i,j\in Q_0$
  we will denote by~$Q_*(i,j)$ the set of paths from~$i$ to~$j$.
  We have a bijection between oriented cycles in $Q^\flat$ and words $p_n\cdots p_1$ in paths $p_\nu \in Q_*(0,\infty) \cup Q_*(\infty,0)$.
  Let $p_n\cdots p_1$ be such a word, and let $c \in Q^\flat_*(0\infty,0\infty)$ be the associated cycle.
  Then, for any representation $N \in \representationvariety{Q^\flat,\mathbf{d}^\flat}$ we have
  \begin{equation}
    \varphi_c(N) = \operatorname{tr}(N_{p_n}\cdots N_{p_1}) = N_{p_n}\cdots N_{p_1} = f_{p_n}(N)\cdots f_{p_1}(N).
  \end{equation}
  The group $\Gm$ acts on $\varphi_c$ with weight
  \begin{equation}
    w = \#\{ r \in \{1,\ldots,n\} \mid p_r \in Q_*(0,\infty) \} - \#\{ s \in \{1,\ldots,n\} \mid p_s \in Q_*(\infty,0) \}
  \end{equation}
  As $Q$ is acyclic, at least one of the sets $Q_*(0,\infty)$ and $Q_*(\infty,0)$ is empty.
  So either $w = n$ if $Q_*(\infty,0) = \emptyset$, or $w = -n$ if $Q_*(0,\infty) = \emptyset$.

  For $\varphi_c$ to be a semi-invariant function of weight $\delta_\infty - \delta_0$,
  we must have $w=1$.
  So this can only occur when there are no oriented paths from $\infty$ to $0$ and $n=1$.
  A semi-invariant function of weight $\delta_\infty - \delta_0$ is therefore
  a linear combination of the functions $f_p$ with $p \in Q_*(0,\infty)$, as claimed.
\end{proof}

We will use this in the following.

\begin{proof}[Proof of \cref{theorem:global-sections}]
  Let $\overline{X} = \modulispace[\stable{\overline{\theta}}]{\overline{Q},\overline{\mathbf{d}}}$ be the doubly-framed moduli space at $i$ and $j$.
  Using the construction from \cref{proposition:annoying-special-cases-ample-stability}
  we let $X' = \modulispace[\stable{\theta'}]{Q',\mathbf{d}'}$.
  We know from \cref{corollary:double-framed-global-sections} and \cref{corollary:amply-stable-global-sections} that
  \begin{equation}
    \label{equation:big-diagram}
    \begin{tikzcd}[column sep=scriptsize]
      p \arrow[mapsto]{d}{} &[-2em] &[-2em] e_j\field Qe_i \arrow{r}{h_{i,j}^\mathcal{U}} \arrow{d}{} & \HH^0(X,\mathcal{U}_i^\vee \otimes \mathcal{U}_j) \arrow{d}{} &[-2em] s \arrow[mapsto]{d}{} &[-2em] \\
      (j\to\infty)p(0\to i) & q_\infty p'q_0 & e_\infty \field\overline{Q}e_0 \arrow{r}{h_{0,\infty}^{\smash{\overline{\mathcal{U}}}}} & \HH^0(\overline{X},\smash{\overline{\mathcal{U}}}_0^\vee \otimes \smash{\overline{\mathcal{U}}}_\infty) & \smash{\overline{\mathcal{U}}}_{j \to \infty} \circ u^*s \circ \smash{\overline{\mathcal{U}}}_{0 \to i} & \smash{\overline{\mathcal{U}}}_{q_\infty} \circ g^*s' \circ \smash{\overline{\mathcal{U}}}_{q_0}\\
      & p' \arrow[mapsto]{u}{} &[-2em] e_{j'}\field Q'e_{i'} \arrow{r}{h_{i',j'}^{\mathcal{U}'}} \arrow{u}{} & \HH^0(X,\mathcal{U}_{i'}^{\prime\vee} \otimes \mathcal{U}'_{j'}) \arrow{u}{} & & s' \arrow[mapsto]{u}{}
    \end{tikzcd}
  \end{equation}
  is commutative and all vertical maps are isomorphisms.
  To show that the horizontal maps are isomorphisms, it is therefore enough to show that one of them is. We will show that $h_{i',j'}^{\mathcal{U}'}$ is an isomorphism in the following.

  The line bundle $\mathcal{U}_{i'}^{\prime\vee} \otimes \mathcal{U}_{j'}'$
  is the line bundle $\mathcal{L}(\delta_{j'} - \delta_{i'})$.
  The data $Q'$, $\mathbf{d}'$, and $\theta'$ satisfy \cref{assumption:standing-assumption}
  by \cref{proposition:annoying-special-cases-ample-stability}.
  We may therefore apply \cite[Lemma~3.3]{MR4352662},
  and obtain the identification
  \begin{equation}
    \label{equation:isomorphism-1}
    \field[\representationvariety{Q',\mathbf{d}'}]^{\group{\mathbf{d}'},\delta_{j'}-\delta_{i'}}
    =
    \HH^0(X',\mathcal{L}(\delta_{j'} - \delta_{i'})).
  \end{equation}
  Note that, in loc.~cit., it is required that $\mathbf{d}'$ is $\theta'$-coprime.
  This assumption may be relaxed:
  it is enough to assume that the stable and semi-stable locus for $\mathbf{d}'$ agree,
  which is the case by \cref{proposition:semistability-double-framed}.

  From \cref{lemma:semi-invariant-functions} we have the explicit isomorphism
  \begin{equation}
    \label{equation:isomorphism-2}
    e_{j'}\field Q' e_{i'}
    \overset{\simeq}{\to}
    \field[\representationvariety{Q',\mathbf{d}'}]^{\group{\mathbf{d}'},\delta_{j'}-\delta_{i'}}.
  \end{equation}
  The composition of \eqref{equation:isomorphism-1} and \eqref{equation:isomorphism-2} is the map $h_{i',j'}^{\mathcal{U}'}$.
  Thus the horizontal maps in \eqref{equation:big-diagram} are all isomorphisms,
  and through the identifications in the diagram
  we conclude that~$h_{i,j}^{\mathcal{U}}$ is an isomorphism,
  which concludes the proof of the first part of the theorem.

  The isomorphism \eqref{equation:algebra-isomorphism} follows because the algebra structure
  on~$\field Q$ is given in terms of the composition of paths,
  and on~$\End_X(\mathcal{U})$ in terms of the composition of morphisms.
\end{proof}

\subsection{Global sections of the tangent bundle}
For smooth projective quiver moduli there is,
e.g., by \cite[Lemma~4.2]{2307.01711v2},
the 4-term exact sequence
\begin{equation}
  \label{equation:4-term-sequence}
  0
  \to \mathcal{O}_X
  \xrightarrow{\Phi} \bigoplus_{i \in Q_0} \mathcal{U}_i^\vee \otimes \mathcal{U}_i
  \xrightarrow{\Psi} \bigoplus_{a \in Q_1} \mathcal{U}_{\source(a)}^\vee \otimes \mathcal{U}_{\target(a)}
  \to \tangent_X
  \to 0.
\end{equation}
The morphism~$\Phi$ is the standard inclusion,
and the morphism~$\Psi = \sigma_{\mathcal{U},\mathcal{U}}$ is defined in Section~3.1 of op.~cit.
Over an open subset $V \subseteq X$, the map $\Psi$ sends a section $f = (f_i)_{i \in Q_0}$
of~$\mathcal{U}_{\source(a)}^\vee\otimes\mathcal{U}_{\target(a)}\cong\sheafHom(\mathcal{U}_{\source(a)},\mathcal{U}_{\target(a)})$
over $V$ to
\begin{equation}
  \label{equation:Psi-description}
  (f_{\target(a)} \circ \mathcal{U}_a|_V - \mathcal{U}_a|_V \circ f_{\source(a)})_{a \in Q_1}.
\end{equation}

For the description of the vector fields
we recall the main result of \cite{rigidity},
which is also used to prove \cref{theorem:admissible}.
\begin{theorem}[Cohomology vanishing]
  \label{theorem:cohomology-vanishing}
  Let~$Q,\mathbf{d}$ and~$\theta$ be as in \cref{assumption:standing-assumption},
  and assume in addition that~$\mathbf{d}$ is~$\theta$-strongly amply stable.
  Consider~$X=\modulispace[\stable\theta]{Q,\mathbf{d}}$.
  For all~$i,j\in Q_0$ we have
  \begin{equation}
    \HH^{\geq 1}(X,\mathcal{U}_i^\vee\otimes\mathcal{U}_j)=0.
  \end{equation}
\end{theorem}
In op.~cit.~it is used to establish the following (infinitesimal) rigidity result.
\begin{corollary}[Rigidity]
  \label{corollary:rigidity}
  Let~$Q,\mathbf{d}$ and~$\theta$ be as in \cref{assumption:standing-assumption},
  and assume in addition that~$\mathbf{d}$ is~$\theta$-strongly amply stable.
  Consider~$X=\modulispace[\stable\theta]{Q,\mathbf{d}}$.
  Then
  \begin{equation}
    \HH^{\geq1}(X,\tangent_X)=0.
  \end{equation}
\end{corollary}

We can now give the description \eqref{equation:vector-fields-presentation} of the vector fields.

\begin{proof}[Proof of \cref{theorem:vector-fields}]
  By \cref{theorem:global-sections},
  a global section of $\bigoplus_{i \in Q_0} \mathcal{U}_i^\vee \otimes \mathcal{U}_i$
  is of the form $(z_i \operatorname{id}_{\mathcal{U}_i})_{i \in Q_0}$,
  which by \eqref{equation:Psi-description} is then mapped to $((z_{\target(a)} - z_{\source(a)})\mathcal{U}_a)_{a \in Q_1}$.
  This shows that the induced map of $\Psi$ on global sections corresponds to $\psi$ under the identifications
  \begin{equation}
    \begin{aligned}
      \bigoplus_{i \in Q_0} \mathcal{U}_i^\vee \otimes \mathcal{U}_i &\cong \bigoplus_{i \in Q_0} e_i\field Qe_i \\
      \bigoplus_{a \in Q_1} \mathcal{U}_{\source(a)}^\vee \otimes \mathcal{U}_{\target(a)} &\cong \bigoplus_{a \in Q_1} e_{\target(a)}\field Qe_{\source(a)}.
    \end{aligned}
  \end{equation}
  It is immediate that the map induced by $\Phi$ from \eqref{equation:4-term-sequence}
  on global sections corresponds to the morphism $\phi$ defined in \eqref{equation:phi}.
  This provides us with an exact sequence
  \begin{equation}
    0
    \to\field
    \xrightarrow{\phi} \bigoplus_{i \in Q_0} e_i\field Qe_i
    \xrightarrow{\psi} \bigoplus_{a \in Q_1} e_{\target(a)}\field Qe_{\source(a)}
    \to\HH^0(X,\tangent_X)
    \to\bigoplus_{i \in Q_0} \HH^1(X,\mathcal{U}_i^\vee \otimes \mathcal{U}_i).
  \end{equation}
  The rightmost term vanishes by \cref{theorem:cohomology-vanishing}.
  This proves \cref{theorem:vector-fields}.
\end{proof}

We want to point out a similarity between \cref{theorem:vector-fields}
and how the Hochschild cohomology of the path algebra~$\field Q$ is computed
in \cite[\S1.6]{MR1035222}.
The following statement is obtained by inspecting the proof in~loc.~cit.
\begin{theorem}[Happel]
  \label{theorem:happel}
  Let~$Q$ be an acyclic quiver.
  Then there exists a 4-term exact sequence
  \begin{equation}
    \label{equation:happel}
    0
    \to\field
    \to\bigoplus_{i\in Q_0}\field
    \to\bigoplus_{\alpha\in Q_1}e_{\target(\alpha)}\field Q e_{\source(\alpha)}
    \to\hochschild^1(\field Q)
    \to0.
  \end{equation}
\end{theorem}
This allows us to explain the origin of \cref{conjecture:lie-algebra}.
Recall that the outer automorphism group~$\Out(\field Q)$
is the affine algebraic group~$\Aut(\field Q)/\operatorname{Inn}(\field Q)$,
where~$\operatorname{Inn}(\field Q)$ are the inner automorphisms.
By \cite[Proposition~1.1]{MR1905030}
we have an isomorphism of Lie algebras
\begin{equation}
  \label{equation:hochschild-lie-algebras}
  \hochschild^1(\field Q)\cong\Lie\Out(\field Q),
\end{equation}
where the Lie algebra structure on the left is given by the Gerstenhaber bracket.
On the other hand we have,
for any smooth projective variety~$X$,
and thus in particular for the quiver moduli we are interested in,
an isomorphism of Lie algebras
\begin{equation}
  \HH^0(X,\tangent_X)\cong\Lie\Aut(X),
\end{equation}
where the Lie algebra structure on the left is given by the Schouten--Nijenhuis bracket of vector fields.
This explains the origin of \cref{conjecture:lie-algebra}.

The following example shows how \cref{theorem:global-sections} (and thus \cref{theorem:vector-fields})
fails without ample stability.
\begin{example}
  \label{example:3-vertex}
  Consider the 3-vertex quiver
  \begin{equation}
    \label{equation:3-vertex-quiver}
    Q\colon
    \begin{tikzpicture}[baseline = -20pt, node distance = 1.5cm]
      \node (1)                {};
      \node (2) [right of = 1] {};
      \node (3) [below of = 2] {};
      \draw (1) circle (2pt) node[above] {1};
      \draw (2) circle (2pt) node[above] {2};
      \draw (3) circle (2pt) node[below] {3};
      \draw[->]                  (1) edge (2);
      \draw[->, bend left = 25]  (2) edge (3);
      \draw[->, bend right = 25] (2) edge (3);
      \draw[->]                  (1) edge (3);
    \end{tikzpicture}
  \end{equation}
  for the thin dimension vector~$\mathbf{d}=\mathbf{1}=(1,1,1)$.
  As discussed at the end of \cite{MR4352662},
  there exists an identification
  \begin{equation}
    \modulispace[\stable{\theta}]{Q,\mathbf{1}}
    \cong
    \operatorname{Bl}_p\mathbb{P}^2,
  \end{equation}
  where~$\theta=\theta_\can=(2,1,-3)$ is the canonical stability condition.

  One can compute that there is precisely one other stability chamber
  where the associated moduli space is not empty.
  Let~$\theta'$ be a stability parameter in this chamber,
  e.g., $\theta'=(2,-1,-1)$.
  Then
  \begin{equation}
    \modulispace[\stable{\theta'}]{Q,\mathbf{1}}
    \cong
    \mathbb{P}^2,
  \end{equation}
  because the moduli space is a smooth projective rational surface of Picard rank~1,
  which can be determined by computing the Betti numbers using \cite[Corollary 6.9]{MR1974891},
  as implemented in \cite{hodge-diamond-cutter}.

  Every condition except the ample stability in \cref{assumption:standing-assumption} holds,
  and ample stability fails because the Picard rank drops.
  The summands of the universal representation are line bundles on~$\mathbb{P}^2$.
  But it is impossible to have an isomorphism
  \begin{equation}
    \HH^0(\modulispace[\stable{\theta'}]{Q,\mathbf{1}},\mathcal{U}_2^\vee\otimes\mathcal{U}_3)
    \cong
    e_3\field Qe_2
  \end{equation}
  because the right-hand side is~2\dash dimensional,
  which is impossible for the global sections of the line bundle~$\mathcal{U}_2^\vee\otimes\mathcal{U}_3$ on~$\mathbb{P}^2$,
  because those dimensions are necessarily of the form~$\binom{n+2}{n}$,
  so~$0,1,3,6,\ldots$

  We can also observe the failure of \cref{theorem:vector-fields}.
  Through its identification with~$\mathbb{P}^2$
  we obtain
  \begin{equation}
    \HH^0(\modulispace[\stable{\theta'}]{Q,\mathbf{1}},\tangent_{\modulispace[\stable{\theta'}]{Q,\mathbf{1}}})
    \cong
    \field^8,
  \end{equation}
  whereas
  \begin{equation}
    \hochschild^1(\field Q)
    \cong
    \field^6
  \end{equation}
  by \eqref{equation:happel}.
\end{example}

\section{Functors associated to the universal representation}
\label{section:admissible}

Given a smooth projective variety~$S$,
an acyclic quiver $Q$ and a locally free left $\mathcal{O}_SQ$-module $\mathcal{M}$,
we can consider associated \emph{four} Fourier--Mukai-like functors.
The conditions on~$S$, $Q$ and~$\mathcal{M}$ can be relaxed,
at the cost of not working with the bounded derived category,
but we will not do this.

Usually Fourier--Mukai functors are considered in the context of smooth projective varieties
(without a sheaf of algebras)
where these variations can be ignored,
but because we work with noncommutative algebras
we want to point out how they can be compared.

Using the sheaf Hom we have a covariant and a contravariant version
\begin{equation}
  \begin{aligned}
    \Phi_{\mathcal{M}}
    \colonequals\RsheafHom_{\mathcal{O}_SQ}(\mathcal{M},-\otimes\mathcal{O}_S)
    &:\derived^\bounded(\field Q)\to\derived^\bounded(S) \\
    \Psi_{\mathcal{M}}
    \colonequals\RsheafHom_{\mathcal{O}_SQ}(-\otimes\mathcal{O}_S,\mathcal{M})
    &:\derived^\bounded(\field Q)^\op\to\derived^\bounded(S)
  \end{aligned}
\end{equation}
and using the tensor product we have two covariant versions,
but the first considers \emph{right}~$\field Q$-modules,
leading to
\begin{equation}
  \begin{aligned}
    \Chi_{\mathcal{M}}
    \colonequals-\otimes_{\field Q}^{\mathbf{L}}\mathcal{M}
    &:\derived^\bounded(\field Q^\op)\to\derived^\bounded(S) \\
    \Omega_{\mathcal{M}}
    \colonequals (\dual\mathcal{M})\otimes_{\field Q}^{\mathbf{L}}-
    &:\derived^\bounded(\field Q)\to\derived^\bounded(S),
  \end{aligned}
\end{equation}
where $\dual\mathcal{M} = \mathcal{M}^\vee = \sheafHom_{\mathcal{O}_S}(\mathcal{M},\mathcal{O}_S)$,
equipped with the natural structure of a \emph{right} $\mathcal{O}_SQ$-module.
We will use similar notation for turning a left~$\field Q$-module into a right~$\field Q$-module,
which can also be considered as a left~$\field Q^\op$-module.

Note that all four functors depend on $S$ and $Q$ as well, but we will suppress this in the notation.
Their relationship is explained by the following lemmas.
\begin{lemma}
  \label{lemma:identification-1}
  Let~$N$ be an object in~$\derived^\bounded(\field Q^\op)$.
  Then
  \begin{equation}
    \Chi_{\mathcal{M}}(N)^\vee
    \cong
    \Psi_{\dual\mathcal{M}}(N)
    \cong
    \Phi_{\mathcal{M}}(\dual N)
  \end{equation}
  in~$\derived^\bounded(S)$.
\end{lemma}

\begin{proof}
  By definition we have
  \begin{equation}
    \Chi_{\mathcal{M}}(N)^\vee
    =
    \RsheafHom_{\mathcal{O}_S}(N\otimes_{\field Q}^{\mathbf{L}}\mathcal{M},\mathcal{O}_S).
  \end{equation}
  Using the isomorphism~$N\otimes_{\field Q}^{\mathbf{L}}\mathcal{M} \cong N\otimes_{\field}^{\mathbf{L}} \mathcal{O}_S \otimes_{\mathcal{O}_SQ}^{\mathbf{L}} \mathcal{M}$
  and the tensor-Hom adjunction,
  we can rewrite the right-hand side to
  \begin{equation}
    \RsheafHom_{\mathcal{O}_SQ^\op}(N\otimes_\field^{\mathbf{L}}\mathcal{O}_S,\dual\mathcal{M})
  \end{equation}
  which is~$\Psi_{\dual\mathcal{M}}(N)$,
  resp.~to
  \begin{equation}
    \RsheafHom_{\mathcal{O}_SQ}(\mathcal{M},\dual M\otimes_\field^{\mathbf{L}}\mathcal{O}_S)
  \end{equation}
  which is~$\Psi_{\mathcal{M}}(\dual M)$.
\end{proof}

Similarly we have the following lemma.
\begin{lemma}
  \label{lemma:identification-2}
  Let~$M$ be an object in~$\derived^\bounded(\field Q)$.
  Then
  \begin{equation}
    \Omega_{\mathcal{M}}(M)^\vee
    \cong
    \Psi_{\mathcal{M}}(M)
    \cong
    \Phi_{\dual\mathcal{M}}(\dual M)
  \end{equation}
  in~$\derived^\bounded(S)$.
\end{lemma}

\subsection{Admissible embeddings}

Following the sheaf-theoretic examples of admissible embeddings cited in the introduction,
we will now prove \cref{theorem:admissible}.
In fact, the admissible embeddings cited in the introduction
are preceded by an admissible embedding
in a restricted setting of quiver moduli,
obtained by Altmann--Hille \cite[Theorem~1.3]{MR1688469},
which we will recall to illustrate the methods.
In a different restricted setting,
that of quiver flag varieties,
it was obtained by Craw--Ito--Karmazyn in \cite[Example~2.9]{MR3803802}.

Their condition that the canonical stability parameter~$\theta_\can$ does not lie on
any~$(1,0)$- or~$(t,t)$-walls in the terminology of \cite[\S2.2]{MR1688469}
is implied by \cref{assumption:standing-assumption},
where~$\mathbf{d}=\mathbf{1}$.

\begin{theorem}[Altmann--Hille]
  \label{theorem:altmann-hille}
  Let~$Q,\mathbf{d}=\mathbf{1}$ and~$\theta_\can$ satisfy \cref{assumption:standing-assumption}.
  Consider~$X=\modulispace[\stable{\theta_\can}]{Q,\mathbf{1}}$.
  Then~$X$ is a smooth projective toric Fano variety,
  with~$\rk\Pic X=\#Q_0-1$,
  and
  \begin{equation}
    \label{equation:altmann-hille-embedding}
    \Chi_{\mathcal{U}}\colon\derived^\bounded(\field Q^\op)\to\derived^\bounded(X)
  \end{equation}
  is fully faithful.
\end{theorem}

In the thin case
the proof reduces to
\begin{itemize}
  \item higher cohomology vanishing for the tensor products~$\mathcal{U}_i^\vee\otimes\mathcal{U}_j$ \cite[Theorem~3.6]{MR1688469};
  \item an identification of the global sections~$\HH^0(\mathcal{U}_i^\vee\otimes\mathcal{U}_j)\cong e_j\field Qe_i$ \cite[Theorem~4.3]{MR1688469}.
\end{itemize}
The vanishing is an application of the Kodaira vanishing theorem,
for which it is important that~$\mathcal{U}_i^\vee\otimes\mathcal{U}_j$ is a line bundle
and that~$X$ is a Fano variety,
and hence why \cref{theorem:altmann-hille} is stated only for~$\theta_\can$.
The identification of the global sections uses the toric description in \cite[Proposition~3.1]{MR1688469}.
A minor but important detail which is omitted in the proof of \cite[Theorem~4.3]{MR1688469}
is the fact that the isomorphism needs to be induced by the functor \eqref{equation:altmann-hille-embedding}.
We will address this in our more general setting in the proof of \cref{theorem:admissible}.

\paragraph{Admissible embedding in the general case}
To check the fully faithfulness in \cref{theorem:admissible}
we will apply the following fully faithfulness criterion \cite[Proposition~1.49]{MR2244106}.

\begin{proposition}
  \label{proposition:fully-faithfulness-criterion}
  Let~$F\colon\mathcal{C}\to\mathcal{D}$ be an exact functor between triangulated categories,
  which admits a left and a right adjoint.
  Let~$\mathcal{S}$ be a spanning class for~$\mathcal{C}$ such that
  for all~$C,C'\in\mathcal{S}$ and all~$i\in\mathbb{Z}$
  the natural morphism
  \begin{equation}
    F_{C,C'}\colon\Hom_{\mathcal{C}}(C,C'[i])\to\Hom_{\mathcal{D}}(F(C),F(C'[i]))
  \end{equation}
  is an isomorphism.
  Then~$F$ is fully faithful.
\end{proposition}
We will apply \cref{proposition:fully-faithfulness-criterion}
using the first spanning class in the following standard lemma.
\begin{lemma}
  \label{lemma:spanning-classes}
  Let~$Q$ be an acyclic quiver.
  Then the following are spanning classes:
  \begin{itemize}
    \item the set~$\{P_i\mid i\in Q_0\}$ of indecomposable projectives;
    \item the set~$\{I_i\mid i\in Q_0\}$ of indecomposable injectives.
  \end{itemize}
\end{lemma}

We will apply \cref{proposition:fully-faithfulness-criterion}
using the spanning class of indecomposable projectives from \cref{lemma:spanning-classes}.

\begin{proof}[Proof of \cref{theorem:admissible}]
  Let us first show that $\Psi_\mathcal{U}$ is fully faithful.
  As $P_i \otimes \mathcal{O}_X$ is a projective left $\mathcal{O}_XQ$-module, we have
  \begin{equation}
    \Psi_{\mathcal{U}}(P_i) = \RsheafHom_{\mathcal{O}_XQ}(P_i \otimes \mathcal{O}_X, \mathcal{U}) \cong \sheafHom_{\mathcal{O}_XQ}(P_i \otimes \mathcal{O}_X, \mathcal{U}) \cong\mathcal{U}_i.
  \end{equation}
  The isomorphism $\sheafHom_{\mathcal{O}_XQ}(P_i \otimes \mathcal{O}_X, \mathcal{U}) \cong\mathcal{U}_i$ can be checked affine locally. Let $U = \operatorname{Spec} A$ and let $M$ be the left $AQ$-module belonging to $\mathcal{U}|_U$. Over $U$ we have the obvious isomorphism $\Hom_{AQ}(P_i \otimes A,M) \cong M_i$. They glue to a global isomorphism over $X$.

  We have
  \begin{itemize}
    \item $\Hom_{\field Q}(P_j,P_i)\cong e_j\field Qe_i$,
    \item $\Ext_{\field Q}^n(P_j,P_i)=0$ for all~$n\geq 1$, because~$P_j$ is projective.
  \end{itemize}
  So we need to show that
  \begin{itemize}
    \item the natural map
      \begin{equation}
        \label{equation:Phi-Pi-Pj}
        \Psi_{\mathcal{U},P_i,P_j}\colon\Hom_{\field Q}(P_j,P_i)\to\Hom_X(\mathcal{U}_i,\mathcal{U}_j)\cong\HH^0(X,\mathcal{U}_i^\vee\otimes\mathcal{U}_j)
      \end{equation}
      is an isomorphism;
    \item $\Ext_X^n(\mathcal{U}_i,\mathcal{U}_j)=0$ for all~$n\geq 1$.
  \end{itemize}
  The second point follows from the cohomology vanishing in \cite{rigidity},
  recalled in \cref{theorem:cohomology-vanishing}.

  For the first point, let us consider the basis of~$e_j\field Qe_i$ of paths from~$i$ to~$j$.
  Let~$q$ be such a path.
  The associated homomorphism $\psi_q\colon P_j \to P_i$ is given on basis elements
  by mapping a path $p$ starting at $j$ to $\psi_q(p) = pq$.
  Using this, we see that the diagram
  \begin{equation}
    \begin{tikzcd}
      \sheafHom_{\mathcal{O}_XQ}(P_i \otimes \mathcal{O}_X,\mathcal{U}) \arrow{r}{\cong} \arrow{d}{\sheafHom(\psi_q \otimes \operatorname{id},-)} & \mathcal{U}_i \arrow{d}{\mathcal{U}_q} \\
      \sheafHom_{\mathcal{O}_XQ}(P_j \otimes \mathcal{O}_X,\mathcal{U}) \arrow{r}{\cong} & \mathcal{U}_j
    \end{tikzcd}
  \end{equation}
  commutes; this can again be checked affine locally.
  The image of~$\psi_q$ under~$\Psi_{\mathcal{U},P_i,P_j}$
  is therefore the morphism~$\mathcal{U}_q\colon\mathcal{U}_i\to\mathcal{U}_j$.
  This shows that the composition
  \begin{equation}
    \begin{tikzcd}
      e_j\field Qe_i \arrow{r}{\cong} & \Hom_{\field Q}(P_j,P_i) \arrow{r}{\Psi_{\mathcal{U},P_i,P_j}} &[1em] \HH^0(X,\mathcal{U}_i^\vee \otimes \mathcal{U}_j) \\[-2em]
      q \arrow[mapsto]{r}{} & \psi_q
    \end{tikzcd}
  \end{equation}
  agrees with the isomorphism from \cref{theorem:global-sections}. Therefore $\Psi_{\mathcal{U},P_i,P_j}$ must be an isomorphism as well.

  The proof of the fully faithfulness of $\Phi_\mathcal{U}$ is similar.
  Here we use that $\Phi_\mathcal{U}(I_i) \cong \mathcal{U}_i^\vee$.
  The comparison in \cref{lemma:identification-1,lemma:identification-2}
  gives the fully faithfulness of~$\Chi_{\mathcal{U}}$ and~$\Omega_{\mathcal{U}}$.
\end{proof}

\subsection{Rigidity and vector fields from the admissible embedding}
In \cite[\S4]{MR3950704} the fully faithful functor~$\Phi_{\mathcal{I}}$
from \eqref{equation:universal-ideal-sheaf-fully-faithful}
(provided~$\mathcal{O}_S$ is exceptional)
is used to relate the Hochschild cohomology of~$S$
to the deformation theory of~$\operatorname{Hilb}^nS$.
We will now explain how a similar reasoning
allows us to describe~$\HH^i(X,\tangent_X)$,
where~$X$ still denotes~$\modulispace[\stable\theta]{Q,\mathbf{d}}$.
Because the proof of \cref{theorem:admissible}
is heavily dependent on \cref{theorem:cohomology-vanishing} and \cref{theorem:global-sections}
this is not an independent description of~$\HH^i(X,\tangent_X)$
(i.e., the combination of \cref{theorem:global-sections} and \cref{corollary:rigidity}),
but it highlights an important parallel between the behavior of different moduli problems.

Recall from \cite[Equation~(24)]{2307.01711v2}
the local-to-global spectral sequence
\begin{equation}
  \label{equation:spectral-sequence}
  \mathrm{E}_2^{p,q}
  =\HH^p(X,\sheafExt_{\mathcal{O}_XQ}^p(\mathcal{U},\mathcal{U}))
  \Rightarrow
  \Ext_{\mathcal{O}_XQ}^{p+q}(\mathcal{U},\mathcal{U}).
\end{equation}
The following lemma identifies the abutment of \eqref{equation:spectral-sequence}
with the Hochschild cohomology of~$\field Q$.
It is the analogue of \cite[Lemmas~16 and~17]{MR3950704}.
We have opted to phrase it using a variation of the functor~$\Chi_{\mathcal{U}}$,
but it can also be phrased using any of the other three functors.
\begin{lemma}
  \label{lemma:hochschild-cohomology-is-Ext}
  The functor
  \begin{equation}
    \label{equation:diagonal-embedding}
    -\otimes_{\field Q}^{\mathbf{L}}\mathcal{U}\colon\derived^\bounded(\field Q\otimes\field Q^\op)\to\derived^\bounded(\mathcal{O}_XQ)
  \end{equation}
  is fully faithful,
  and sends the diagonal~$\field Q$-bimodule~$\field Q$ to~$\mathcal{U}$.
  In particular, there exists an isomorphism of vector spaces
  \begin{equation}
    \hochschild^\bullet(\field Q)\cong\Ext_{\mathcal{O}_XQ}^\bullet(\mathcal{U},\mathcal{U}).
  \end{equation}
\end{lemma}

\begin{proof}
  The right adjoint to the functor \eqref{equation:diagonal-embedding}
  is~$\RHom_{\mathcal{O}_X}(\mathcal{U},-)$,
  and to check fully faithfulness of \eqref{equation:diagonal-embedding}
  we will check that the unit of the adjunction is a natural equivalence.
  Let $M$ be a complex of finitely generated~$\field Q$-$\field Q$-bimodules. We want to show that the natural morphism in $\derived^\bounded(\field Q \otimes \field Q^\op)$
  \begin{equation}
    \label{equation:unit-adjunction}
    M \to \RHom_{\mathcal{O}_X}(\mathcal{U},M \otimes_{\field Q}^\mathbf{L} \mathcal{U})
  \end{equation}
  is an isomorphism.
  We have established in \cref{theorem:admissible} (or rather, in its proof) that~$\Chi_\mathcal{U}$
  is fully faithful.
  Its right adjoint is $\RHom_{\mathcal{O}_X}(\mathcal{U},-)\colon \derived^\bounded(X) \to \derived^\bounded(\field Q^\op)$.
  So we know that \eqref{equation:unit-adjunction} is an isomorphism after applying
  the forgetful functor to~$\derived^\bounded(\field Q^\op)$.
  As this functor reflects isomorphisms,
  it was an isomorphism already in~$\derived^\bounded(\field Q\otimes_\field\field Q^\op)$,
  and thus the unit of the adjunction is an isomorphism.

  The natural isomorphism~$\field Q\otimes_{\field Q}^{\mathbf{L}}\mathcal{U}\cong\mathcal{U}$
  identifies the image of the diagonal bimodule with the universal representation~$\mathcal{U}$.
  Denoting the functor in \eqref{equation:diagonal-embedding} by~$F$,
  the isomorphism of vector spaces is given by composing the isomorphism
  \begin{equation}
    F_{\field Q,\field Q}\colon
    \Ext_{\field Q\otimes\field Q^\op}^\bullet(\field Q,\field Q)
    \to
    \Ext_{\mathcal{O}_XQ}^\bullet(\mathcal{U},\mathcal{U})
  \end{equation}
  with the standard isomorphism~$\Ext_{\field Q\otimes\field Q^\op}^\bullet(\field Q,\field Q)\cong\hochschild^\bullet(\field Q)$.
\end{proof}
Now we turn out attention to the objects on the~$\mathrm{E}_2$-page.
We have the following description.
\begin{lemma}
  \label{lemma:relative-Ext}
  We have
  \begin{equation}
    \sheafExt_{\mathcal{O}_XQ}^i(\mathcal{U},\mathcal{U})
    \cong
    \begin{cases}
      \mathcal{O}_X & i=0 \\
      \tangent_X & i=1 \\
      0 & i\geq 2.
    \end{cases}
  \end{equation}
\end{lemma}

\begin{proof}
  The first isomorphism is given by the natural morphism~$\mathcal{O}_X\to\sheafHom_{\mathcal{O}_XQ}(\mathcal{U},\mathcal{U})$,
  which can be checked to be an isomorphism fiberwise,
  because stable representations are simple.
  The second isomorphism is \cite[Proposition~3.7(2)]{2307.01711v2}.
  The vanishing for~$i\geq 2$ can be checked fiberwise,
  using that~$\Ext^{\geq 2}$ vanishes for any representation of~$Q$ over~$\field$.
\end{proof}

The following proposition proves how admissibility gives
the promised identification between the Hochschild cohomology of~$\field Q$
and deformation theory of~$X$.
We do not need to require strong ample stability,
it suffices that~$\mathcal{U}$ exists and gives a fully faithful functor.
\begin{proposition}
  \label{proposition:spectral-sequence}
  Let~$Q$, $\mathbf{d}$ and~$\theta$ be as in \cref{assumption:standing-assumption}.
  Let~$X=\modulispace[\stable\theta]{Q,\mathbf{d}}$.
  Assume that~$\Phi_{\mathcal{U}}$ is fully faithful.
  Then
  \begin{equation}
    \HH^i(X,\tangent_X)
    \cong
    \begin{cases}
      \hochschild^1(\field Q) & i=0 \\
      0 & i\geq 1.
    \end{cases}
  \end{equation}
\end{proposition}

\begin{proof}
  Because~$X$ is a smooth projective rational variety,
  as shown in \cite[Theorem 6.4]{MR1914089},
  we have that~$\HH^{\geq1}(X,\mathcal{O}_X)=0$ and~$\HH^0(X,\mathcal{O}_X)\cong\field$.
  This fact,
  together with \cref{lemma:relative-Ext}
  shows that the~$\mathrm{E}_2$-page of the spectral sequence \eqref{equation:spectral-sequence}
  looks like
  \begin{equation}
    \label{equation:E2-page}
    \begin{tikzcd}
      \vdots & \vdots              & \vdots & \iddots \\
      0      & \HH^2(X,\tangent_X) & 0      & \ldots \\
      0      & \HH^1(X,\tangent_X) & 0      & \ldots \\
      \field & \HH^0(X,\tangent_X) & 0      & \ldots
    \end{tikzcd}
  \end{equation}
  From \cref{lemma:hochschild-cohomology-is-Ext} we know the abutment of the spectral sequence,
  and we see that~$\HH^{\geq 1}(X,\tangent_X)$ needs to be cancelled in the spectral sequence, as~$\hochschild^{\geq 2}(\field Q)=0$.
  But this is impossible, because the spectral sequence necessarily already degenerates on the~$\mathrm{E}_2$-page,
  thus they are zero to begin with.
  Similarly we obtain an isomorphism
  \begin{equation}
    \HH^0(X,\tangent_X)\cong\hochschild^1(\field Q).
  \end{equation}
\end{proof}

\printbibliography

\emph{Pieter Belmans}, \url{pieter.belmans@uni.lu} \\
Department of Mathematics, Universit\'e de Luxembourg, 6, avenue de la Fonte, L-4364 Esch-sur-Alzette, Luxembourg

\emph{Ana-Maria Brecan}, \url{anabrecan@gmail.com}

\emph{Hans Franzen}, \url{hans.franzen@math.upb.de} \\
Institute of Mathematics, Paderborn University, Warburger Stra\ss e 100, 33098 Paderborn, Germany

\emph{Markus Reineke}, \url{markus.reineke@rub.de} \\
Fakultat f\"ur Mathematik, Ruhr-Universit\"at Bochum, Universit\"atsstra\ss e 150, 44780 Bochum, Germany

\end{document}